\documentclass[11pt]{article}
\usepackage{amsmath,amssymb,amsthm}
\usepackage[a4paper,margin=1in]{geometry}
\usepackage[colorlinks=true,linkcolor=blue,citecolor=blue,urlcolor=blue]{hyperref}
\allowdisplaybreaks
\numberwithin{equation}{section}
\newtheorem{theorem}{Theorem}[section]
\newtheorem{lemma}[theorem]{Lemma}
\newtheorem{proposition}[theorem]{Proposition}
\newtheorem{corollary}[theorem]{Corollary}
\theoremstyle{definition}\newtheorem{assumption}[theorem]{Assumption}
\newtheorem{example}[theorem]{Example}
\theoremstyle{remark}

\title{Positive normalized solutions for a singular regularized $p(x)$-Laplacian Dirichlet problem}
\author{Mustafa Avci\\Faculty of Science and Technology, Applied Mathematics, Athabasca University\\ \text{mavci@athabascau.ca (primary), avcixmustafa@gmail.com}}
\date{}
\begin{document}\maketitle

\begin{abstract}
We study a singular Dirichlet problem driven by a regularized $p(x)$-Laplacian under a prescribed modular constraint. We construct nonnegative constrained minimizers for nonsingular approximating problems and, under the stated variational and boundary-control hypotheses, obtain a positive normalized weak solution as the reaction regularization vanishes. The approximation preserves the prescribed modular and yields subsequential strong Sobolev convergence of the states together with convergence of their constraint multipliers, while the gradient regularization remains fixed. Finally, we present a normalized singular problem on the unit disk as an illustrative example.
\end{abstract}


\section{Introduction}\label{sec:introduction}

The purpose of this paper is to construct positive normalized weak solutions of the singular $\varepsilon$-regularized $p(x)$-Laplacian Dirichlet problem
\begin{equation}\label{eq:main-eigen-system}
\left\{
\begin{aligned}
-\mathrm{div} \left((\varepsilon^2+|\nabla u|^2)^{ \frac{p(x)-2}{2}}\nabla u\right)&=\lambda u^{p(x)-1}+\mu h(x)u^{-\gamma(x)},\\
u&>0 \text{ in }\Omega,\,\, u=0\text{ on }\partial\Omega,\\
J(u)&=\int_\Omega\frac{u^{p(x)}}{p(x)}dx=m.
\end{aligned}
\right.
\end{equation}
Here $\Omega\subset\mathbb{R}^N$, $N\geq 2$, is a bounded domain with Lipschitz boundary $\partial\Omega$; $p$ is log-H\"older continuous on $\overline\Omega$; $\varepsilon,m>0$,  $\mu \geq 0$ are fixed parameters; the parameter $\lambda=\lambda_{\varepsilon,m,\mu} \in \mathbb{R}$ is not prescribed in advance, it arises as the Lagrange multiplier associated with the modular constraint; $h\in L^\infty(\Omega)$ is a nonnegative function; and $\gamma\in C(\overline\Omega)$ with $0<\gamma<1$.\\

The prescribed-modular formulation is motivated by the failure of the usual constant-exponent homogeneity when the exponent varies with position. More clearly, if $p$ is constant, then
\[
J(tu)=t^pJ(u),
\]
and the amplitude of a nonzero state can be normalized by a simple power scaling. If $p=p(x)$, however, then
\[
J(tu)=\int_\Omega \frac{t^{p(x)}|u(x)|^{p(x)}}{p(x)} dx,
\]
and no single power of $t$ can be factored out. Similarly, the variable exponent Rayleigh-type quotient is not homogeneous in the usual sense, and its lower spectral behavior may differ substantially from the constant exponent case (see, e.g., \cite{FanZhangZhao,FranzinaLindqvist}). Thus, the normalization here is a \emph{modular constraint}, rather than a prescribed value of the reaction coefficient $\lambda$. More explicitly, fixing $J(u)=m$ selects an amplitude, and $\lambda$ is the multiplier
associated with that choice. It need not be positive when $\mu>0$ and is
not identified with a positive first eigenvalue.\\

\noindent
The singular reaction $h(x)u^{-\gamma(x)}$, $0<\gamma(x)<1$, creates a second difficulty. This term becomes unbounded as $u$ approaches zero. Since the homogeneous Dirichlet condition forces the solution to vanish at the boundary, positivity and sufficient integrability of this term are not available at the outset. We therefore do not attack the target problem directly. Instead, for each $\delta>0$, we replace the singular factor $u^{-\gamma(x)}$ by the nonsingular approximation $(u+\delta)^{-\gamma(x)}$, and minimize the corresponding approximate constrained functional.\\

\noindent
A third difficulty appears when the approximation parameter tends to zero. Almost-everywhere convergence $u_\delta \rightarrow u$ does not by itself justify the passage
\[
h(x)(u_\delta+\delta)^{-\gamma(x)}\rightarrow h(x)u^{-\gamma(x)}
\]
against \textit{arbitrary} test functions in $W^{1,p(\cdot)}_0(\Omega)$, especially near the boundary where the Dirichlet solution vanishes. The relevant boundary information is the uniform boundary barrier  stated explicitly in Assumption \ref{ass:barrier}. This uniform boundary barrier condition along with the variable exponent Hardy inequality leads to the uniform dual estimate
\[
\left|\int_\Omega h(x)(u_\delta+\delta)^{-\gamma(x)} \phi dx \right| \leq C \|\phi\|.
\]
This estimate used as the principal mechanism that makes the singular limit meaningful for every Sobolev test function. Once the boundary condition (in Assumption \ref{ass:barrier}) is available, the variable exponent Hardy inequality, the Hölder inequality, compactness, and the $(S_+)$ property of the regularized operator yield preservation of the modular mass, positivity of the limiting state, global convergence of the singular terms, and strong convergence of the gradients.\\

The study of equations with variable exponents has attracted considerable attention over the past two decades due to their important role in various areas of mathematics, including the calculus of variations and partial differential equations \cite{cruz2013variable,diening2011lebesgue}. Moreover, such equations arise naturally in a wide range of physical and engineering applications, including the modeling of electrorheological fluids \cite{ruzicka2007electrorheological}, the analysis of non-Newtonian fluids \cite{zhikov1997meyer}, fluid flow through porous media \cite{amaziane2009nonlinear}, magnetostatics \cite{cekic2012lp}, image restoration \cite{chen2006variable}, and capillarity phenomena \cite{avci2013ni}. We also refer the reader to \cite{Boureanu,heidarkhani2022critical,yucedag2015existence} and the references therein for further developments and applications.\\

\noindent
For singular problems involving variable exponents, the literature is extensive and cannot be reviewed exhaustively here. Instead, we highlight several foundational works and recent studies that are particularly relevant to the present paper.\\

\noindent
Fan, Zhang, and Zhao \cite{FanZhangZhao} show that the Dirichlet $p(x)$-Laplacian spectrum can have either zero or positive infimum under different hypotheses. Franzina and Lindqvist \cite{FranzinaLindqvist} instead derive an eigenvalue problem from a homogeneous quotient of two Luxemburg norms. These formulations should not be conflated with \eqref{eq:main-eigen-system}. Here the prescribed quantity is $\int_\Omega |u|^{p(x)}/p(x)dx x=m$, not a Luxemburg-norm quotient, and the singular perturbation is part of the constrained energy. The factor $1/p(x)$ in $J$ yields precisely $u^{p(x)-1}$ in its first variation. Our objective is a positive normalized weak state and its approximation, not a description of the nonlinear spectrum, simplicity of an eigenvalue, or identification of a first eigenvalue.\\

\noindent
Zhang \cite{ZhangSingular} obtains positive solutions of Dirichlet $p(x)$-Laplacian equations with a decreasing singular reaction and an additional power term, using subsolution and supersolution methods. The present problem \eqref{eq:main-eigen-system} instead couples $0<\gamma(x)<1$ to a fixed modular constraint and retains a positive gradient-regularization parameter. Thus its distinct question is whether approximate constrained states, together with their multipliers, converge to a globally testable singular equation.\\

\noindent
Papageorgiou, R\u{a}dulescu, and Sun \cite{PapageorgiouRadulescuSun} prove the existence of at least two smooth positive solutions for a nonparametric anisotropic problem combining singular and superlinear effects. Their variational, truncation, and comparison techniques address multiplicity and regularity. Neither conclusion is sought in the present manuscript, whose results remain in
$W_0^{1,p(\cdot)}(\Omega)$ and whose extra boundary information is stated as an assumption. Thus, the present work is not a regularity improvement of the existing singular theory.\\

\noindent
Majdak and Papageorgiou \cite{MajdakPapageorgiou} study singular equations with unbalanced $(p(x),q(x))$-growth, an unbounded source coefficient, and a singular exponent below one only near the boundary. They establish a unique bounded positive solution through approximation and use a weak formulation with all admissible Sobolev tests. Their argument also uses Hardy estimates and $(S_+)$ compactness. Our problem differs through the single regularized flux and the fixed modular constraint, which introduce an unknown multiplier.\\

\noindent
The broader theory of singular elliptic equations shows that negative powers of the unknown occur naturally in stationary reaction--diffusion models and lead to distinctive boundary, integrability, and regularity questions (see, e.g., \cite{Avci-1,Avci-2,Ghanmi,HernandezManceboVega,Lu,OlivaPetitta}).\\

\noindent
The problem studied here differs in its variational architecture. The solution is selected on the nonnegative modular manifold $\mathcal M_m^+$, the multiplier is generated by the modular constraint, and the singular equation is obtained from constrained nonsingular approximations. Consequently, the limiting argument must establish more than positivity and weak solvability: it must simultaneously preserve the exact value $J(u)=m$, control the sequence of multipliers, and identify the singular source against every test function in $W_0^{1,p(\cdot)}(\Omega)$.\\

The remainder of the paper is organized as follows. Section~\ref{Auxiliary} collects and establishes the auxiliary results needed for the constrained approximation. After stating the baseline assumptions and the additional boundary-barrier and weight hypothesis, it records the variable exponent Hardy inequality, the growth and differentiability properties of the regularized energy, and quantitative monotonicity estimates for the flux. Uniform control of the singular primitives and continuity of the modular constraint then lead to equi-coercivity and existence of approximate minimizers. The section also derives the Euler multiplier under the stated admissibility of signed constrained variations. Section~\ref{sec:limit} addresses passage to the singular reaction at fixed $\varepsilon>0$, $m>0$, and $\mu\ge0$. It first explains how the boundary barrier and the weight condition make the singular sources uniformly bounded in the Sobolev dual space and permit convergence against every fixed Sobolev test function. It then establishes the $(S_+)$ property of the regularized operator and presents the main singular-limit statement for the approximate minimizing pairs, including preservation of the modular constraint, positivity of the limiting state, and convergence of the states and multipliers. Lastly, as an illustrative example, we provide a normalized singular problem posed on the unit disk and demonstrate that it satisfies hypotheses of the main result of the paper.

\subsection{Functional framework}
We start with some basic concepts of variable Lebesgue-Sobolev spaces. For more details, and the proof of the following propositions, we refer the reader to \cite{cruz2013variable,diening2011lebesgue,fan2001spaces,RadulescuRepovs}.
\begin{equation*}
C_{+}\left(\overline{\Omega}\right) =\left\{p:\,p\in C\left(\overline{\Omega}\right),\text{ }\inf p\left( x\right)>1\text{ for all } x\in \overline{\Omega}\right\} .
\end{equation*}
For $p\in C_{+}(\overline{\Omega})$ denote
\begin{equation}\label{p}
p^-:=\min_{x\in\overline\Omega}p(x) \leq p(x) \leq \max_{x\in\overline\Omega}p(x)=:p^+ < \infty.
\end{equation}
For a measurable function $u:\Omega\longrightarrow\mathbb R$,
\begin{equation}\label{eq:scalar-variable-modular}
\rho_{p(\cdot)}(u):=\int_\Omega |u(x)|^{p(x)} dx
\end{equation}
defines a convex modular on $L^{p(\cdot)}(\Omega)$.\\
\noindent
The variable exponent Lebesgue space is defined by
\begin{equation*}
L^{p(\cdot)}(\Omega):= \left\{ u:\Omega\longrightarrow\mathbb R \text{ measurable}: \rho_{p(\cdot)}(u)<\infty \right\}.
\end{equation*}
The norm defined on $L^{p(\cdot)}(\Omega)$ is given by
\begin{equation*}
\|u\|_{L^{p(\cdot)}(\Omega)} = \|u\|_{p(\cdot)}:=\inf \left\{ \tau>0: \rho_{p(\cdot)} \left(\frac{u}{\tau} \right) \leq1 \right\}
\end{equation*}
and called the Luxemburg norm. Equipped with this norm, $L^{p(\cdot)}(\Omega)$ is a Banach space.\\

\begin{proposition}\label{Prop:2.1}
For any $u\in L^{p(\cdot) }(\Omega)$ and $v\in L^{p^{\prime}(\cdot)}(\Omega)$, we have
\begin{equation*}
\int_{\Omega}|uv|dx\leq C(p^{-},(p^{-})^{\prime})\|u\|_{p(\cdot)}\|v\|_{p^{\prime}(\cdot)}
\end{equation*}
where $L^{p^{\prime}(\cdot)}(\Omega) $ is the conjugate space of $L^{p(\cdot)}(\Omega)$ such that $\frac{1}{p(x)}+\frac{1}{p^{\prime}(x)}=1$.
\end{proposition}

\begin{proposition} \label{Prop:2.1-a}
If $p(x)>1$ for all $x\in \overline{\Omega }$ and $\frac{1}{p(x) }+\frac{1}{p^{\prime}(x)}=1$, then
for all $\zeta,\eta\in \mathbb{R}^{N}$
\begin{equation*}
|\zeta||\eta|\leq \frac{1}{p(x)}|\zeta|^{p(x)}+\frac{1}{p^{\prime}(x)}|\eta|^{p^{\prime}(x)}.
\end{equation*}
\end{proposition}

\begin{proposition}\label{Prop:2.2}
If $u,u_{n}\in L^{p(\cdot)}(\Omega)$ ($n=1,2,...$), we have
\begin{itemize}
\item[\rm(i)] $\|u\|_{p(\cdot)}<1 ( =1;>1) \Leftrightarrow \rho_{p(\cdot)}(u) <1 (=1;>1)$;
\item[\rm(ii)] $\|u\|_{p( \cdot)}>1 \implies \|u\|_{p(\cdot)}^{p^{-}}\leq \rho_{p(\cdot)}(u) \leq \|u\|_{p( \cdot)}^{p^{+}}$;
\item[\rm(iii)]$\|u\|_{p(\cdot)}\leq1 \implies \|u\|_{p(\cdot) }^{p^{+}}\leq \rho_{p(\cdot)}(u) \leq \|u\|_{p(\cdot)}^{p^{-}}$;
\item[\rm(iv)] $\lim\limits_{n\rightarrow \infty }\|u_{n}\|_{p(\cdot)}=0\Leftrightarrow \lim\limits_{n\rightarrow \infty}\rho_{p(\cdot)}
(u_{n})=0;\lim\limits_{n\rightarrow \infty}\|u_{n}\|_{p(\cdot)}=\infty \Leftrightarrow \lim\limits_{n\rightarrow \infty}\rho_{p(\cdot)}
(u_{n})=\infty.$
\end{itemize}
\end{proposition}

\begin{proposition}\label{Prop:2.3}
If $u,u_{n}\in L^{p(\cdot)}(\Omega)$ ($n=1,2,...$), then the following statements are
equivalent:
\begin{itemize}
\item[\rm(i)] $\lim\limits_{n\rightarrow \infty}\|u_{n}-u\|_{p(\cdot)}=0$;
\item[\rm(ii)] $\lim\limits_{n\rightarrow \infty }\rho_{p(\cdot)} (u_{n}-u)=0$;
\item[\rm(iii)] $u_{n}\rightarrow u\mathit{\ }$\textit{in measure in}$\mathit{\ }\Omega \mathit{\ }$\textit{and}
$\mathit{\ }\lim\limits_{n\rightarrow \infty}\rho_{p(\cdot)} (u_{n})=\rho_{p(\cdot)}(u)$.
\end{itemize}
\end{proposition}

\noindent
The variable exponent Sobolev space $W^{1,p(\cdot)}(\Omega)$ is defined as
\[
W^{1,p(\cdot)}(\Omega) = \left\{ u \in L^{p(\cdot)}(\Omega): |\nabla u| \in L^{p(\cdot)}(\Omega) \right\}
\]
and equipped with the norm
\[
\|u\|_{W^{1,p(\cdot)}(\Omega)}:= \|u\|_{p(\cdot)} + \|\nabla u\|_{p(\cdot)}.
\]
The space $W^{1,p(\cdot)}_0(\Omega)$ is defined as the closure of $C_0^\infty(\Omega)$ with respect to the $W^{1,p(\cdot)}(\Omega)$ norm. Provided $p$ satisfies the log-H\"older continuity condition, and $|\Omega|<\infty$, the Poincar\'e inequality holds on $W^{1,p(\cdot)}_0(\Omega)$. Consequently, the gradient norm provides an equivalent norm on $W^{1,p(\cdot)}_0(\Omega)$; that is,
\[
\|u\|_{W^{1,p(\cdot)}_0(\Omega)}=\|u\|:=\|\nabla u\|_{p(\cdot)}\quad \forall u \in W^{1,p(\cdot)}_0(\Omega).
\]
We will show the conjugate space of $W^{1,p(\cdot)}_0(\Omega)$  by $(W^{1,p(\cdot)}_0(\Omega))^*$ and the corresponding norm by
\[
\|\cdot\|_ {(W^{1,p(\cdot)}_0(\Omega))^*}= \|\cdot\|_ {*}.
\]
\begin{proposition}\label{Prop:2.4} If \eqref{p} holds, then the spaces
$L^{p(\cdot)}( \Omega)$ and $W^{1,p(\cdot)}( \Omega)$ are separable and reflexive Banach spaces.
\end{proposition}

\begin{proposition}\label{Prop:2.5}  If $\Omega$ has finite measure, and $q(x)\leq p(x)$ almost everywhere in  $\Omega$, then the embedding
$L^{p(\cdot)}(\Omega) \hookrightarrow L^{q(\cdot)}(\Omega)$ is continuous.
\end{proposition}

\begin{proposition}\label{Prop:2.6} Let $q\in C(\overline{\Omega })$. If $1\leq q(x) <p^{\ast }(x)$ for all $x\in \overline{\Omega }$, then the embedding $W^{1,p(\cdot)}(\Omega) \hookrightarrow L^{q(\cdot)}(\Omega)$ is compact and continuous, where
\[p^{\ast }(x) =
\begin{cases}
\frac{Np(x)}{N-p(x)}, & \mbox{if }p(x)<N, \\
+\infty,  & \mbox{if } p(x)\geq N.
\end{cases}\]
\end{proposition}

\section{The Variational Structure and Auxiliary Results}\label{Auxiliary}

\subsection{Assumptions} \label{assumptions}
Throughout the paper, we assume that the following hypotheses.
\begin{assumption} \label{ass:base}
The following conditions are satisfied:
\begin{enumerate}
\item[\textnormal{\rm(H1)}] The exponent $p:\overline\Omega\longrightarrow(1,\infty)$ satisfies $1 <p^-\leq p(x) \leq p^+ < \infty$.
\item[\textnormal{\rm(H2)}] $p$ is log-H\"older continuous on $\overline\Omega$. More precisely, there exists a constant $C_{\log}>0$ such that
\begin{equation*}
|p(x)-p(y)| \leq \frac{C_{\log}} {\log\!\left(e+\dfrac{1}{|x-y|} \right)} \qquad \text{for all } x,y\in\overline\Omega, \quad x\neq y.
\end{equation*}
\item[\textnormal{\rm(H3)}] $h\in L^\infty(\Omega)$ and $h\geq0$ almost everywhere in $\Omega$,\quad $h\not\equiv0$.
\item[\textnormal{\rm(H4)}] $\gamma\in C(\overline\Omega)$ and
\begin{equation*}
0 <\gamma^-\leq \gamma(x) \leq  \gamma^+ < 1 \quad \text{for every }x\in\overline\Omega.
\end{equation*}
\end{enumerate}

\end{assumption}
\begin{assumption}\label{ass:barrier}
Let $u_\delta = u_{\varepsilon,m,\mu,\delta} \in \mathcal M_m^+$ be a minimizer obtained in Lemma \ref{thm:approx}, and define
\begin{equation}\label{eq:boundary-distance}
d(x):= \operatorname{dist}(x,\partial\Omega), \quad x\in\Omega. 
\end{equation}
The following additional conditions are satisfied:
\begin{enumerate}
\item[\textnormal{\rm(H5)}] There exist constants $\delta_0\in(0,1], \, c_*>0, \, \vartheta>0$, independent of $\delta$, such that
\begin{equation*}\label{hyp:boundary-barrier}
u_\delta(x) \geq c_*d(x)^\vartheta \quad \text{for almost every }x\in\Omega \text{ and every }0<\delta\leq\delta_0. 
\end{equation*}
\item[\textnormal{\rm(H6)}] The boundary weight
\begin{equation*}\label{eq:boundary-weight}
\omega_\vartheta(x) := h(x)d(x)^{1-\vartheta\gamma(x)} 
\end{equation*}
satisfies
\begin{equation*}\label{hyp:weight}
\omega_\vartheta \in L^{p'(\cdot)}(\Omega),
\end{equation*}
\end{enumerate}
where the constants $\delta_0, \, c_*, \, \vartheta$ may depend on $\Omega,\varepsilon,m,\mu,p,h$ and $\gamma$, but they are independent of $\delta\in(0,\delta_0]$.
\end{assumption}

\subsection{The variational structure}\label{variational-structure}
We distinguish two regularization parameters. The fixed parameter $\varepsilon>0$ regularizes the gradient-dependent diffusion, whereas the auxiliary parameter $\delta>0$ regularizes the singular reaction and will later tend to zero.\\
Under Assumption \ref{ass:base}, fix $\varepsilon,m>0$ and $\mu\geq0$.
Set $X:=W_0^{1,p(\cdot)}(\Omega)$,  and define the convex cone
\begin{equation}\label{cone}
X_+:=\{u\in X:v\geq0\text{ a.e. in }\Omega\},
\end{equation}
For $\xi\in\mathbb R^N$, define
\begin{equation}\label{regularizatio}
r_\varepsilon(\xi):= \left( \varepsilon^2+|\xi|^2 \right)^{1/2},
\end{equation}
\begin{equation}\label{reg-flux-b}
A_\varepsilon(x,\xi):= r_\varepsilon(\xi)^{p(x)-2}\xi,
\end{equation}
and \begin{equation}\label{reg-density}
\Phi_\varepsilon(x,\xi):= \frac{ r_\varepsilon(\xi)^{p(x)} - \varepsilon^{p(x)} }{ p(x) }.
\end{equation}
The associated regularized energy is
\begin{equation}\label{reg-func}
E_\varepsilon(u):= \int_\Omega \Phi_\varepsilon(x,\nabla u) dx, \quad u\in X.
\end{equation}
The density satisfies
\[
\nabla_\xi\Phi_\varepsilon(x,\xi) = A_\varepsilon(x,\xi).
\]
The prescribed normalization is defined by the modular constraint
\begin{equation}\label{eq:normalized-modular-functional}
J(u) = \int_\Omega \frac{|u|^{p(x)}}{p(x)} dx, \quad u\in X.
\end{equation}
For $m>0$, set
\begin{equation}\label{eq:prescribed-modular-constraint}
\mathcal M_m := \left\{ u\in X: J(u)=m \right\}
\end{equation}
and
\begin{equation}\label{eq:positive-prescribed-modular-constraint}
\mathcal M_m^+:= \mathcal M_m  \cap X_+.
\end{equation}
Formally, the constrained singular functional associated with the target singular problem \eqref{eq:main-eigen-system} is
\begin{equation}\label{regular-modular-functional}
I_{\mu}(u)=I_{\mu,0}(u):= E_\varepsilon(u) - \mu \int_\Omega h(x) \frac{u^{1-\gamma(x)}}{1-\gamma(x)} dx, \quad u\in X_+. 
\end{equation}
Although the singular primitive is finite on $X_+$, its derivative contains $u^{-\gamma(x)}$, which is singular at $u=0$. We therefore introduce, for $0<\delta\leq1$,
\begin{equation}\label{reg-sing-func}
G_\delta(x,s):= \frac{(s+\delta)^{1-\gamma(x)} - \delta^{1-\gamma(x)} }{ 1-\gamma(x)}, \quad s\geq0. 
\end{equation}
Then
\[
G_\delta(x,0)=0 \quad\text{and}\quad \partial_sG_\delta(x,s)= (s+\delta)^{-\gamma(x)}.
\]
The approximate constrained functional is then defined by
\begin{equation}\label{normal-modular-functional}
I_{\mu,\delta}(u):= E_\varepsilon(u)-\mu \int_\Omega h(x)G_\delta(x,u) dx, \quad u\in X_+.
\end{equation}
For each fixed $\delta>0$, a minimizer $u_\delta$ of $I_{\mu,\delta}$ over $\mathcal M_m^+$, whenever signed constrained variations are admissible, has an associated Lagrange multiplier $\lambda_\delta\in\mathbb R$ and satisfies
\begin{equation}\label{eq:epsilon-delta-approximate-weak-problem}
\int_\Omega A_\varepsilon(x,\nabla u_\delta) \cdot\nabla\phi dx = \lambda_\delta \int_\Omega u_\delta^{p(x)-1}\phi dx + \mu \int_\Omega h(x) (u_\delta+\delta)^{-\gamma(x)} \phi dx,\quad \phi\in X.
\end{equation}
Thus nonnegativity restricts the admissible states, while the weak identity is tested in the full space $X$.\\
Note also that the target singular problem \eqref{eq:main-eigen-system} is obtained formally by letting $\delta\downarrow0$. Its weak formulation is to find
$(u,\lambda) \in X \times\mathbb R$
such that
\[
u>0 \quad\text{almost everywhere in }\Omega, \quad J(u)=m,
\]
and
\begin{equation}\label{eq:target-singular-weak-problem}
\int_\Omega A_\varepsilon(x,\nabla u) \cdot\nabla\phi dx = \lambda \int_\Omega u^{p(x)-1}\phi dx + \mu \int_\Omega h(x)u^{-\gamma(x)}\phi dx,\quad \phi\in X.
\end{equation}
Therefore $\varepsilon>0$ remains fixed, while the auxiliary parameter $\delta>0$ is used to approximate and subsequently recover the singular reaction as $\delta\downarrow0$.

\subsection{Auxiliary results}\label{Auxiliary Results}
\begin{lemma} \label{lem:hardy}
Suppose that Assumption \ref{ass:base} holds. Then there exists a constant $C_H>0$ such that
\begin{equation}\label{hardy-eq}
\left\|\frac{u}{d} \right\|_{p(\cdot)} \leq C_H \| u\|
\end{equation}
for every $u \in X$.
\end{lemma}
\begin{proof}
Since $\Omega$ is a bounded Lipschitz domain, it satisfies a uniform local exterior measure-density condition. More precisely, there exist constants $b_0>0$ and $r_0>0$ such that
\[
|B(z,r)\cap\Omega^c|\geq b_0 |B(z,r)|
\]
for every $z\in\partial\Omega$ and every $0<r\leq r_0$.\\
We first observe that, because $\Omega$ is bounded, this estimate can be extended to all $r>0$. To see this, let $\omega_N:=|B(0,1)|$ and set
\[
R:=\max\left\{r_0,\left(\frac{2|\Omega|}{\omega_N}\right)^{1/N}\right\}.
\]
For $r_0<r\leq R$, the inclusion $B(z,r_0)\subset B(z,r)$ gives
\[
\begin{aligned}
|B(z,r)\cap\Omega^c|\geq |B(z,r_0)\cap\Omega^c| \geq b_0\omega_Nr_0^N \geq b_0\left(\frac{r_0}{R}\right)^N |B(z,r)|.
\end{aligned}
\]
On the other hand, if $r\geq R$, then
\[
\begin{aligned}
|B(z,r)\cap\Omega^c|\geq|B(z,r)|-|\Omega| \geq \frac12 |B(z,r)|.
\end{aligned}
\]
Consequently, with
\[
b_\Omega:=\min\left\{b_0\left(\frac{r_0}{R}\right)^N,\frac12\right\},
\]
we have
\[
|B(z,r)\cap\Omega^c|\geq b_\Omega |B(z,r)|
\]
for every $z\in\partial\Omega$ and every $r>0$. Moreover, for $0<|x-y|\leq 1/2$,
\[
\log\left(e+\frac1{|x-y|}\right)\geq-\log|x-y|,
\]
and hence Assumption~\ref{ass:base} implies
\[
|p(x)-p(y)|\leq \frac{C_{\log}}{-\log|x-y|}.
\]
Thus $p$ satisfies the log-H\"older continuity condition used in Theorem 3.3 of \cite{HarjulehtoHardy}. Together with \rm(H1), all the hypotheses of Theorem~3.3 are satisfied. Therefore applying Theorem 3.3 with $a=0$ and $\delta(x)=d(x)$ yields \eqref{hardy-eq} for every $u \in X$, where $C_H>0$ depends only on $p$, $N$, and $b_\Omega$.
\end{proof}

\begin{lemma}\label{lem:energy}
Under Assumption \ref{ass:base}, for every $\varepsilon>0$ and for almost every $x\in\Omega$, $\Phi_\varepsilon(x,\cdot)$ is $C^2$ and strictly convex. There are constants $C_1,C_2,C_3>0$, depending only on $\varepsilon,p^-,p^+$, such that
\begin{equation}\label{growth-phi}
\frac1{p^+}|\xi|^{p(x)}-C_1\leq\Phi_\varepsilon(x,\xi)\leq C_2(1+|\xi|^{p(x)}),
\end{equation}
and the mapping $\eta \mapsto A_\varepsilon(\cdot,\eta(\cdot))$ induces a Nemytskii operator from $L^{p(\cdot)}(\Omega)$ into $L^{p'(\cdot)}(\Omega)$ with
\begin{equation}\label{growth-flux}
|A_\varepsilon(x,\xi)|\leq C_3(1+|\xi|^{p(x)-1}).
\end{equation}
Consequently $E_\varepsilon$ is coercive, sequentially weakly lower semicontinuous, and continuously G\^ateaux differentiable, with
\[
\langle E_\varepsilon'(u),\phi\rangle=\int_\Omega A_\varepsilon(x,\nabla u)\cdot\nabla\phi dx.
\]
\end{lemma}

\begin{proof}
Fix $x\in\Omega$. Since $\varepsilon>0$, the function $r_\varepsilon(\xi)$ is strictly positive and belongs to $C^\infty(\mathbb{R}^N)$. Thus,
\begin{equation*}
\frac{\partial\Phi_\varepsilon} {\partial\xi_i}(x,\xi) = r_\varepsilon(\xi)^{p(x)-2}\xi_i \quad \text{for each } i\in\{1,\ldots,N\},
\end{equation*}
providing
\begin{equation}\label{grad-a}
\nabla_\xi\Phi_\varepsilon(x,\xi)=r_\varepsilon(\xi)^{p(x)-2}\xi = A_\varepsilon(x,\xi).
\end{equation}
Differentiating once more gives
\begin{equation*}
\frac{\partial}{\partial\xi_j} \left( r_\varepsilon(\xi)^{p(x)-2}\xi_i \right) =(p(x)-2) r_\varepsilon(\xi)^{p(x)-4} \xi_i\xi_j + r_\varepsilon(\xi)^{p(x)-2}\delta_{ij}\quad \text{for } i,j\in\{1,\ldots,N\}.
\end{equation*}
Therefore
\begin{equation}\label{Hessian-a}
D^2_{\xi}\Phi_\varepsilon(x,\xi)=D_{\xi}A_\varepsilon(x,\xi) = r_\varepsilon(\xi)^{p(x)-2}I + \left(p(x)-2\right) r_\varepsilon(\xi)^{p(x)-4} \xi\otimes\xi,
\end{equation}
which is the Hessian of $\Phi_\varepsilon$, where $\xi\otimes\eta = \left( \xi_i\eta_j \right)_{i,j=1}^N$  is the tensor product for vectors $\xi,\eta\in\mathbb{R}^N$.
Because $r_\varepsilon(\xi)>0$, all the factors appearing in the formula \eqref{Hessian-a} are continuous for every $\xi\in\mathbb{R}^N$. Hence $\Phi_\varepsilon(x,\cdot)$ is of class $C^2(\mathbb{R}^N)$.\\
Next, writing the formula \eqref{Hessian-a} in quadratic form gives
\begin{equation}\label{Hessian-aa}
\eta^{\mathsf T} D^2_{\xi}\Phi_\varepsilon(x,\xi)\eta=\eta^{\mathsf T}D_{\xi}A_\varepsilon(x,\xi)\eta = r_\varepsilon(\xi)^{p(x)-2}|\eta|^2 + (p(x)-2) r_\varepsilon(\xi)^{p(x)-4} (\xi\cdot\eta)^2.
\end{equation}
If $\xi\neq0$, decompose $\eta$ into its component parallel to $\xi$ and its component orthogonal to $\xi$:
\[
\eta=\eta_\parallel+\eta_\perp \quad \text{where} \quad \eta_\parallel:= \frac{\xi\cdot\eta}{|\xi|^2}\xi \quad \text{and} \quad \eta_\perp:= \eta-\eta_\parallel.
\]
Then
\[
\eta_\perp\cdot\xi=0, \quad |\eta|^2 = |\eta_\perp|^2+|\eta_\parallel|^2 \quad \text{and} \quad (\xi\cdot\eta)^2 = |\xi|^2|\eta_\parallel|^2.
\]
Substituting these identities gives
\begin{equation*}
\eta^{\mathsf T} D^2_{\xi}\Phi_\varepsilon(x,\xi)\eta = r_\varepsilon(\xi)^{p(x)-2} |\eta_\perp|^2 + r_\varepsilon(\xi)^{p(x)-2} \left(1+ (p(x)-2) \frac{|\xi|^2} {\varepsilon^2+|\xi|^2} \right) |\eta_\parallel|^2.
\end{equation*}
Set
\[
\theta_\varepsilon(\xi):= \frac{|\xi|^2} {\varepsilon^2+|\xi|^2}.
\]
Since $\varepsilon>0$, $0\leq\theta_\varepsilon(\xi)<1$.
Moreover,
\[
1+(p(x)-2)\theta_\varepsilon(\xi) = 1-\theta_\varepsilon(\xi)+(p(x)-1)\theta_\varepsilon(\xi).
\]
The expression on the right is a convex combination of $1$ and $p(x)-1$. Since $p(x)>1$, both numbers are strictly positive. Therefore,
\[
1+(p(x)-2)\theta_\varepsilon(\xi) \geq \min\{1,p(x)-1\}>0.
\]
It follows that, for every $\eta\neq0$,
\begin{equation*}
\eta^{\mathsf T} D^2_{\xi}\Phi_\varepsilon(x,\xi)\eta \geq r_\varepsilon(\xi)^{p(x)-2} \min\{1,p(x)-1\} \left(|\eta_\perp|^2+|\eta_\parallel|^2 \right) = r_\varepsilon(\xi)^{p(x)-2} \min\{1,p(x)-1\} |\eta|^2>0.
\end{equation*}
If $\xi=0$, then
\[
D^2_{\xi}\Phi_\varepsilon(x,0) = \varepsilon^{p(x)-2}I,
\]
and hence
\[
\eta^{\mathsf T} D^2_{\xi}\Phi_\varepsilon(x,0)\eta = \varepsilon^{p(x)-2}|\eta|^2 >0
\]
for every $\eta\neq0$. Thus the Hessian of $\Phi_\varepsilon$ is positive definite at every $\xi\in\mathbb{R}^N$. Consequently, $\Phi_\varepsilon(x,\cdot) $ is strictly convex.\\
On the other hand, since $r_\varepsilon(\xi)^{p(x)} \geq |\xi|^{p(x)}$, we have
\begin{equation*}
\Phi_\varepsilon(x,\xi) \geq \frac{|\xi|^{p(x)}}{p(x)}-\frac{\varepsilon^{p(x)}}{p(x)}.
\end{equation*}
Fix a point $x\in\Omega$ for which $q=p(x)$. Then the function $q\mapsto \frac{\varepsilon^q}{q}$ is continuous on the compact interval $[p^-,p^+]$. Hence
\[
C_1:= \max_{q\in[p^-,p^+]} \frac{\varepsilon^q}{q} <\infty.
\]
It follows that
\begin{equation}\label{lowbnd}
\Phi_\varepsilon(x,\xi) \geq \frac{1}{p^+}|\xi|^{p(x)}-C_1.
\end{equation}
For the upper bound, observe that
\[
\Phi_\varepsilon(x,\xi) \leq \frac{r_\varepsilon(\xi)^{p(x)}}{p(x)} \leq \frac{1}{p^-} \left(\varepsilon^2+|\xi|^2\right)^{p(x)/2}.
\]
For $a,b\geq0$ and $s>0$, using the elementary inequality
\begin{equation}\label{elem-ineq}
(a+b)^s \leq 2^{\max\{s-1,0\}} (a^s+b^s)
\end{equation}
with $a=\varepsilon^2,\, b=|\xi|^2,\, s=\frac{p(x)}{2}$, there is a constant $C_0>0$ depending only on $p^+$, such that
\[
\left(\varepsilon^2+|\xi|^2\right)^{p(x)/2} \leq C_0 \left(\varepsilon^{p(x)}+|\xi|^{p(x)} \right).
\]
Furthermore,
\[
\varepsilon^{p(x)} \leq \max\{\varepsilon^{p^-},\varepsilon^{p^+}\}=:C_\varepsilon.
\]
Consequently,
\begin{equation}\label{upbnd}
\Phi_\varepsilon(x,\xi) \leq \frac{C_0}{p^-} \left(C_\varepsilon+|\xi|^{p(x)} \right) \leq C_2 \left( 1+|\xi|^{p(x)} \right)
\end{equation}
for a suitable constant $C_2>0$.\\
By definition,
\[
|A_\varepsilon(x,\xi)| = \left(\varepsilon^2+|\xi|^2\right)^{(p(x)-2)/2} |\xi|.
\]
To show \eqref{growth-flux}, we distinguish two cases:\\
If $|\xi|\leq1$, then $\varepsilon \leq r_\varepsilon(\xi)=r \leq\sqrt{\varepsilon^2+1}$. Since the map $(r,q)\mapsto r^{q-2}$ is continuous on the compact set
$[\varepsilon,\sqrt{\varepsilon^2+1}] \times[p^-,p^+]$, there exists $C_{\varepsilon,p}>0$ such that
\[
r_\varepsilon(\xi)^{p(x)-2} \leq C_{\varepsilon,p}.
\]
Therefore,
\[
|A_\varepsilon(x,\xi)| \leq C_{\varepsilon,p}|\xi| \leq C_{\varepsilon,p}.
\]
If $|\xi|>1$, then
\[
r_\varepsilon(\xi) = \sqrt{\varepsilon^2+|\xi|^2} \leq \sqrt{\varepsilon^2+1}\,|\xi|.
\]
If $p(x)\geq2$, this yields
\begin{equation*}
|A_\varepsilon(x,\xi)| = r_\varepsilon(\xi)^{p(x)-2}|\xi| \leq \left(\sqrt{\varepsilon^2+1}\right)^{p(x)-2} |\xi|^{p(x)-1} \leq C_{\varepsilon,p} |\xi|^{p(x)-1}.
\end{equation*}
If $1<p(x)<2$, then $p(x)-2<0$, and $r_\varepsilon(\xi)\geq|\xi|$ gives $r_\varepsilon(\xi)^{p(x)-2} \leq |\xi|^{p(x)-2}$. Thus
\[
|A_\varepsilon(x,\xi)| \leq |\xi|^{p(x)-1}.
\]
Combining the cases $|\xi|\leq 1$ and $|\xi|>1$, for some $C_3>0$, we obtain \eqref{growth-flux}; that is,
\[
|A_\varepsilon(x,\xi)| \leq C_3 \left(1+|\xi|^{p(x)-1} \right).
\]
Therefore $A_\varepsilon$ satisfies the Krasnoselskii growth criteria. The map $\xi \mapsto (\varepsilon^2 + |\xi|^2)^{(p(x)-2)/2}$ is the composition of a continuous positive function with the continuous power function $t \mapsto t^{(p(x)-2)/2}$. Hence it is continuous. The product of a continuous scalar function and a continuous vector function is continuous. Therefore $\xi \mapsto (\varepsilon^2 + |\xi|^2)^{(p(x)-2)/2}\xi$ is continuous on $\mathbb{R}^{N}$ for almost every $x$. Since the variable exponent $p$ is measurable, so is $x \mapsto \frac{p(x)-2}{2}$. Hence for an arbitrary fixed vector $\xi \in \mathbb{R}^{N}$, rewriting the flux  $A_\varepsilon$ using the exponential-logarithm identity shows that $x \mapsto (\varepsilon^2 + |\xi|^2)^{(p(x)-2)/2}\xi$ is measurable on $\Omega$ for every fixed $\xi \in \mathbb{R}^N$. In conclusion, the mapping $\eta \mapsto A_\varepsilon(\cdot,\eta(\cdot))$ induces a Nemytskii operator from $L^{p(\cdot)}(\Omega$ into $L^{p'(\cdot)}(\Omega)$.\\
Now we proceed with the coercivity and sequential weak lower semicontinuity of $E_\varepsilon$.\\
Let $u\in X$ with $\|u\|>1$. Using the lower growth estimate of $\Phi_\varepsilon$ along with the modular--norm inequalities in Proposition \ref{Prop:2.2} gives
\begin{equation*}
E_\varepsilon(u) = \int_\Omega \Phi_\varepsilon(x,\nabla u)dx \geq \frac{1}{p^+} \int_\Omega |\nabla u|^{p(x)}dx - C_1|\Omega| = \frac{1}{p^+} \|u\|^{p^-} - C_1|\Omega|.
\end{equation*}
Therefore, $\|u\| \rightarrow\infty$ implies $E_\varepsilon(u)\rightarrow\infty$. Hence $E_\varepsilon$ is coercive.\\
Let $u_n\in X$ for every $n$ with $u_n\rightharpoonup u$in $X$. Then $\nabla u_n\rightharpoonup\nabla u$ in $L^{p(\cdot)}(\Omega)$. Applying the supporting-hyperplane inequality for convex functions and using \eqref{grad-a} gives
\begin{equation*}
\Phi_\varepsilon(x,\nabla u_n) \geq \Phi_\varepsilon(x,\nabla u) + A_\varepsilon(x,\nabla u) \cdot (\nabla u_n-\nabla u)
\end{equation*}
for almost every $x$. Integrating gives
\begin{equation*}
E_\varepsilon(u_n) \geq E_\varepsilon(u) + \int_\Omega A_\varepsilon(x,\nabla u) \cdot (\nabla u_n-\nabla u)dx.
\end{equation*}
Since $A_\varepsilon(\cdot,\nabla u(\cdot)) \in L^{p'(\cdot)}(\Omega)$, the weak convergence of $(\nabla u_n)$ in $L^{p(\cdot)}(\Omega)$ gives
\[
\int_\Omega A_\varepsilon(x,\nabla u) \cdot (\nabla u_n-\nabla u)dx \rightarrow0.
\]
Lastly, taking the lower limit proves
\[
E_\varepsilon(u) \leq \liminf_{n\to\infty}E_\varepsilon(u_n).
\]
Fix $u,\phi\in X$. For $t\neq0$,
\begin{equation}\label{eq-gat}
\frac{E_\varepsilon(u+t\phi)- E_\varepsilon(u)}{t} = \int_\Omega \frac{\Phi_\varepsilon (x,\nabla u+t\nabla\phi) - \Phi_\varepsilon(x,\nabla u)}{t} dx.
\end{equation}
Considering \eqref{grad-a} and applying the fundamental theorem of calculus in the $\xi$ variable yields
\[
\frac{ \Phi_\varepsilon (x,\nabla u+t\nabla\phi) - \Phi_\varepsilon(x,\nabla u) }{t}= \int_0^1 A_\varepsilon (x,\nabla u+s t\nabla\phi) \cdot\nabla\phi ds.
\]
By continuity of $A_\varepsilon(x,\cdot)$
\[
A_\varepsilon (x,\nabla u+s t\nabla\phi) \rightarrow A_\varepsilon(x,\nabla u)
\]
as $t\to 0$, uniformly with respect to $s\in[0,1]$. For $|t|\leq1$ and $s\in[0,1]$, \eqref{growth-flux} gives
\[
\left| A_\varepsilon (x,\nabla u+s t\nabla\phi) \cdot\nabla\phi \right| \leq C \left( 1+ |\nabla u+s t\nabla\phi|^{p(x)-1} \right) |\nabla\phi|
\leq C \left( 1+ (|\nabla u|+|\nabla\phi|)^{p(x)-1} \right) |\nabla\phi|.
\]
Using the uniform power inequality,
\[
(a+b)^{p(x)-1} \leq C \left( 1+a^{p(x)-1}+b^{p(x)-1} \right),
\]
where $C>0$ is some constant not depending on $a, b$, we find
\[
\left| A_\varepsilon (x,\nabla u+s t\nabla\phi) \cdot\nabla\phi \right| \leq C \left(|\nabla\phi| + |\nabla u|^{p(x)-1}|\nabla\phi| + |\nabla\phi|^{p(x)} \right).
\]
Since the right-hand side is integrable due to the standard arguments, i.e. obtained by employing the relevant embeddings, and Propositions \ref{Prop:2.1-a}, \ref{Prop:2.2}, the Dominated Convergence Theorem (DCT)  yields
\[
\lim_{t\to0} \frac{E_\varepsilon(u+t\phi)- E_\varepsilon(u)}{t} = \int_\Omega A_\varepsilon(x,\nabla u) \cdot\nabla\phi dx.
\]
Hence $E_\varepsilon$ is Gâteaux differentiable with
\[
\langle E_\varepsilon'(u),\phi\rangle = \int_\Omega A_\varepsilon(x,\nabla u) \cdot\nabla\phi dx.
\]
Suppose now that $ u_n\rightarrow u$ in $X$. Then
\[
\nabla u_n\rightarrow\nabla u \quad\text{in } L^{p(\cdot)}(\Omega),
\]
and hence
\[
\nabla u_n\rightarrow\nabla u \quad\text{in measure in}\,\Omega.
\]
Since $A_\varepsilon$ is a Carathéodory function (as shown above), it reads
\[
A_\varepsilon(x, \nabla u_n) \rightarrow A_\varepsilon(x, \nabla u)\, \text{ in measure in }\Omega.
\]
Strong convergence in $L^{p(\cdot)}(\Omega)$ implies modular convergence and uniform integrability of the family
$\left\{ |\nabla u_n|^{p(\cdot)} \right\}_{n}$.  Therefore by \eqref{growth-flux} and the fact that $\Omega$ is bounded, the family $\left\{ |A_\varepsilon(x,\nabla u_n)|^{p'(\cdot)} \right\}_{n}$ is uniformly integrable. Thus, by  Vitali's theorem (see, e.g. \cite{Bogachev}), we obtain
\begin{equation}\label{lem2.2-catedory}
A_\varepsilon(x,\nabla u_n) \rightarrow A_\varepsilon(x,\nabla u)\,\text{ in } L^{p'(\cdot)}(\Omega).
\end{equation}
On the other hand, for any $\phi\in X$, Hölder's inequality gives
\[
\begin{aligned}
\left| \langle E_\varepsilon'(u_n) - E_\varepsilon'(u),\phi\rangle \right| &= \left| \int_\Omega \left[ A_\varepsilon(x,\nabla u_n) - A_\varepsilon(x,\nabla u) \right] \cdot\nabla\phi dx \right| \\ &\leq 2 \left\| A_\varepsilon(\cdot,\nabla u_n) - A_\varepsilon(\cdot,\nabla u) \right\|_{p'(\cdot)} \|\phi\|.
\end{aligned}
\]
Taking the supremum over all $\phi$ satisfying $\|\phi\|\leq1$ yields
\[
\|E_\varepsilon'(u_n)-E_\varepsilon'(u)\|_ {*} \leq 2 \left\| A_\varepsilon(\cdot,\nabla u_n) - A_\varepsilon(\cdot,\nabla u) \right\|_{p'(\cdot)} \rightarrow 0,
\]
which implies that $E_\varepsilon'$ is continuous. Therefore $E_\varepsilon$ is continuously Gâteaux differentiable.
\end{proof}

\begin{lemma}\label{lem:quantitative-monotonicity-flux}
There exists a constant $c_0:=c_0(p^-,p^+)>0$ such that, for any $x\in\Omega$, $\xi,\zeta\in\mathbb{R}^N$ fixed the following inequalities hold:
\begin{equation} \label{lema3.2-e}
\left(A_\varepsilon(x,\xi) - A_\varepsilon(x,\zeta) \right) \cdot(\xi-\zeta) \geq c_0|\xi-\zeta|^{p(x)} \quad\text{if } p(x)\geq2,
\end{equation}
and
\begin{equation} \label{lema3.2-f}
\left( A_\varepsilon(x,\xi) - A_\varepsilon(x,\zeta) \right) \cdot(\xi-\zeta)\geq c_0 |\xi-\zeta|^2 \left( \varepsilon+|\xi|+|\zeta| \right)^{p(x)-2} \quad\text{if } 1<p(x)<2.
\end{equation}
In particular,
\[
\left(A_\varepsilon(x,\xi) - A_\varepsilon(x,\zeta) \right) \cdot(\xi-\zeta) > 0
\]
whenever $\xi\neq\zeta$. Thus $A_\varepsilon(x,\cdot)$ is strictly monotone.
\end{lemma}

\begin{proof}
In the following discussions, we fix a point $x\in\Omega$ for which \rm(H1) holds. Define
\[
\theta(t):=(1-t)\zeta+t\xi= \zeta+t\vartheta, \quad \vartheta:=\xi-\zeta, \quad 0\leq t\leq1, \quad \zeta,\xi \in\mathbb R^N.
\]
Then $\theta(0)=\zeta$ and $\theta(1)=\xi$. Then
\begin{align}\label{eq:quantitative-monotonicity-segment}
A_{\varepsilon}(x,\xi) - A_{\varepsilon}(x,\zeta) = \int_0^1 \frac{ d}{dt} A_{\varepsilon}(x,\theta(t)) dt = \int_0^1 D_\xi A_{\varepsilon}(x,\theta(t)) \vartheta  dt.
\end{align}
Taking the Euclidean inner product with $\vartheta$, we obtain
\begin{align}\label{eq:quantitative-monotonicity-segment-form}
\left(A_{\varepsilon}(x,\xi) - A_{\varepsilon}(x,\zeta) \right) \cdot \vartheta = \int_0^1 \vartheta^\top D_\xi A_{\varepsilon}(x,\theta(t)) \vartheta dt.
\end{align}
This identity reduces the monotonicity estimate to a lower bound for the Hessian $D_\xi A_{\varepsilon}$ along the line segment joining $\zeta$ and $\xi$. Thus, we will use the quadratic formula \eqref{Hessian-aa}. The behavior of the second term in the quadratic formula \eqref{Hessian-aa} differs according to whether $p(x)\geq2$) or $1<p(x)<2$. We therefore treat these two cases separately.\\

\noindent
\textit{Case $p(x)\geq 2$.} In this case, we can write
\begin{equation} \label{eq:superquadratic-Hessian-lower}
\vartheta^\top D_\xi A_{\varepsilon,}(x,\theta(t)) \vartheta \geq r_\varepsilon(\theta(t))^{p(x)-2}|\vartheta|^2.
\end{equation}
Since
\[
r_\varepsilon(\theta(t)) = \left( \varepsilon^2+|\theta(t)|^2 \right)^{1/2} \geq |\theta(t)|,
\]
and $p(x)-2\geq0$, we also have
\[
r_\varepsilon(\theta(t))^{p(x)-2} \geq |\theta(t)|^{p(x)-2}.
\]
Therefore,
\begin{equation}\label{eq:superquadratic-Hessian-unregularized-lower}
\vartheta^\top D_\xi A_{\varepsilon}(x,\theta(t)) \vartheta \geq |\theta(t)|^{p(x)-2}|\theta(t)|^2.
\end{equation}
Applying this in \eqref{eq:quantitative-monotonicity-segment-form} with $\theta(t)=\zeta+t\vartheta$ we obtain
\begin{align}\label{eq:superquadratic-segment-integral}
\left(A_{\varepsilon}(x,\xi) - A_{\varepsilon}(x,\zeta) \right) \cdot \vartheta \geq |\vartheta|^2 \int_0^1 |\zeta+t\vartheta|^{p(x)-2} dt.
\end{align}
Notice that if $\vartheta=0$,  then $\xi=\zeta$, and the desired inequality is immediate. Hence assume $\vartheta\neq0$.  Let $e=\frac{\vartheta}{|\vartheta|}$. We decompose $\zeta$ into components parallel and perpendicular to $e$
\[
\zeta = \zeta_\parallel+\zeta_\perp,
\]
where
\[
\zeta_\parallel = (\zeta\cdot e)e \quad \text{and} \quad \zeta_\perp\cdot e=0.
\]
Write $a = \frac{\zeta\cdot e}{|\vartheta|}$. Then
\[
\zeta_\parallel+t\vartheta = |\vartheta|(a+t)e.
\]
Since the parallel and perpendicular components are orthogonal,
\begin{align*}
|\zeta+t\vartheta|^2 = |\zeta_\perp|^2 + |\vartheta|^2|a+t|^2 \geq |\vartheta|^2|a+t|^2.
\end{align*}
Therefore,
\[
|\zeta+t\vartheta| \geq |\vartheta||a+t|.
\]
On the other hand considering that  $p(x)-2\geq 0$, we have
\[
|\zeta+t\vartheta|^{p(x)-2} \geq |\vartheta|^{p(x)-2}|a+t|^{p(x)-2}.
\]
Consequently,
\begin{equation}\label{eq:superquadratic-segment-scalar}
\int_0^1 |\zeta+t\vartheta|^{p(x)-2} dt \geq |\vartheta|^{p(x)-2} \int_0^1 |a+t|^{p(x)-2} dt.
\end{equation}
The scalar integral
\[
a\mapsto \int_0^1 |a+t|^{p(x)-2}  dt
\]
is minimized when the interval $[a,a+1]$ is centered at the origin, namely when $a=-\frac12$. Hence
\begin{align}\label{eq:superquadratic-scalar-lower}
\int_0^1 |a+t|^{p(x)-2} dt &\geq \int_0^1 \left| t-\frac12 \right|^{p(x)-2}  dt = 2 \int_0^{1/2} t^{p(x)-2} dt = \frac{2^{2-p(x)}}{p(x)-1}.
\end{align}
Combining \eqref{eq:superquadratic-segment-scalar} and \eqref{eq:superquadratic-scalar-lower}, we find
\begin{equation}\label{eq:superquadratic-segment-final}
\int_0^1 |\zeta+td|^{p(x)-2} dt \geq \frac{2^{2-p(x)}}{p(x)-1} |\vartheta|^{p(x)-2}.
\end{equation}\label{eq:superquadratic-monotonicity-explicit}
Substituting \eqref{eq:superquadratic-segment-final} and $ \vartheta=\xi-\zeta$ into \eqref{eq:superquadratic-segment-integral} yields
\begin{equation}
\left(A_{\varepsilon}(x,\xi) - A_{\varepsilon}(x,\zeta) \right) \cdot(\xi-\zeta) \geq \frac{2^{2-p(x)}}{p(x)-1} |\xi-\zeta|^{p(x)}\geq \frac{2^{2-p^+}}{p^+-1} |\xi-\zeta|^{p(x)}>0.
\end{equation}

\noindent
\textit{Case $1<p(x)<2$.}  In this case the second term in \eqref{Hessian-aa} is nonpositive, so it cannot simply be discarded. We instead use the regularized tangential and radial eigenvalues of the Hessian $D_\xi A_{\varepsilon}$ to obtain a strictly positive bound. The regularized tangential and radial eigenvalues are
\begin{equation}\label{eq:subquadratic-tangential-eigenvalue}
\mu_\perp(\theta(t)) = r_\varepsilon(\theta(t))^{p(x)-2}
\end{equation}
and
\begin{equation}\label{eq:subquadratic-radial-eigenvalue}
\mu_\parallel(\theta(t)) = r_\varepsilon(\theta(t))^{p(x)-4} \left(\varepsilon^2+(p(x)-1)|\theta(t)|^2 \right),
\end{equation}
respectively. Therefore,
\begin{align*}
\varepsilon^2+(p(x)-1)|\theta(t)|^2 \geq (p(x)-1)\varepsilon^2 + (p(x)-1)|\theta(t)|^2 = (p(x)-1) \left(\varepsilon^2+|\theta(t)|^2 \right).
\end{align*}
It follows that
\begin{align}\label{eq:subquadratic-radial-lower}
\mu_\parallel(\theta(t)) \geq (p(x)-1) r_\varepsilon(\theta(t))^{p(x)-4} \left(\varepsilon^2+|\theta(t)|^2 \right) = (p(x)-1) r_\varepsilon(\theta(t))^{p(x)-2},
\end{align}
and
\[
\mu_\perp(\theta(t)) = r_\varepsilon(\theta(t))^{p(x)-2} \geq (p(x)-1) r_\varepsilon(\theta(t))^{p(x)-2}.
\]
Thus every eigenvalue of $D_\xi A_{\varepsilon}$ is bounded below by $(p(x)-1)r_\varepsilon(\theta(t))^{p(x)-2}$. Consequently,
\begin{equation}\label{eq:subquadratic-Hessian-lower}
\vartheta^\top D_\xi A_{\varepsilon}(x,\theta(t)) \vartheta \geq (p(x)-1) r_\varepsilon(\theta(t))^{p(x)-2}|\vartheta|^2.
\end{equation}
Next, we bound the regularized radius $r_\varepsilon$ along the segment $0\leq t\leq1$. Therefore,
\begin{align*}
0<r_\varepsilon(\theta(t)) = \left(\varepsilon^2+|\theta(t)|^2 \right)^{1/2} \leq \varepsilon+|\theta(t)| \leq \varepsilon+|\xi|+|\zeta|.
\end{align*}
However, since $p(x)-2<0$, we obtain
\begin{equation}\label{eq:subquadratic-radius-lower}
r_\varepsilon(\theta(t))^{p(x)-2} \geq \left(\varepsilon+|\xi|+|\zeta| \right)^{p(x)-2}.
\end{equation}
Using \eqref{eq:subquadratic-radius-lower} in \eqref{eq:subquadratic-Hessian-lower} we obtain \begin{align}\label{eq:subquadratic-segment-lower}
\vartheta^\top D_\xi A_{\varepsilon}(x,\theta(t)) \vartheta\geq (p(x)-1) r_\varepsilon(\theta(t))^{p(x)-2} |\vartheta|^2 \geq (p(x)-1) \left(\varepsilon+|\xi|+|\zeta| \right)^{p(x)-2} |\vartheta|^2.
\end{align}
Integrating over \(t\in[0,1]\) and using \eqref{eq:quantitative-monotonicity-segment-form}, we find
\begin{align}\label{eq:subquadratic-monotonicity-fixed-q}
\left(A_{\varepsilon}(x,\xi) - A_{\varepsilon}(x,\zeta) \right) \cdot \vartheta &\geq  (p(x)-1) |\vartheta|^2 \left(\varepsilon+|\xi|+|\zeta| \right)^{p(x)-2} \int_0^1dt \notag \\&= (p(x)-1) |\vartheta|^2 \left(\varepsilon+|\xi|+|\zeta| \right)^{p(x)-2}.
\end{align}
Therefore,
\begin{align}\label{eq:subquadratic-monotonicity-uniform}
\left(A_\varepsilon(x,\xi) - A_\varepsilon(x,\zeta) \right) \cdot(\xi-\zeta) \geq (p^--1) |\xi-\zeta|^2 \left(\varepsilon+|\xi|+|\zeta| \right)^{p(x)-2},
\end{align}
which is \eqref{lema3.2-f}. Define
\begin{equation}\label{eq:monotonicity-common-constant}
c_0:= \min \left\{ p^--1, \frac{2^{2-p^+}}{p^+-1} \right\}.
\end{equation}
Then both  inequalities \eqref{lema3.2-e} and \eqref{lema3.2-f} hold with same constant $c_0=c_0(p^-,p^+)>0$.\\
\noindent
Lastly, if $\xi\neq\zeta$, then for both cases $1<p(x)<2$ and $p(x)\geq 2$, it holds
\[
\left(A_\varepsilon(x,\xi) - A_\varepsilon(x,\zeta) \right) \cdot(\xi-\zeta) >0.
\]
Hence $A_\varepsilon(x,\cdot)$ is strictly monotone for every $x\in\Omega$. The proof is complete.
\end{proof}

\begin{lemma}\label{lem:primitive}
There is $C_{sing}>0$, independent of $u\in\mathcal M_m^+$ and $\delta \in (0,1]$, such that
\begin{equation}\label{sing-integ}
0\leq\int_\Omega h(x)G_\delta(x,u)dx\leq C_{sing}.
\end{equation}
\end{lemma}

\begin{proof}
Throughout the proof, we let $u\in\mathcal{M}_m^+\, \text{ and }\, 0<\delta\leq1$ be arbitrary.\\
For a fixed $x\in\Omega$ and $s\geq0$, $s+\delta\geq\delta>0$ and the function $t\mapsto t^{1-\gamma(x)} $ is increasing on $[0,\infty)$. Therefore we have
\[
(s+\delta)^{1-\gamma(x)} \geq \delta^{1-\gamma(x)},
\]
which implies that $G_\delta(x,s)\geq 0$ for every $s\geq0$ and almost every \(x\in\Omega\). Because $h\geq 0$ and $u \geq0$  almost everywhere, it follows that
\[
h(x)G_\delta(x,u(x))\geq0 \quad\text{for a.e. }x\in\Omega.
\]
Hence
\[
\int_\Omega h(x)G_\delta(x,u(x))dx\geq0.
\]
Fix $x\in\Omega$, and set $a(x):=1-\gamma(x)$, and hence, $0<a(x)<1$.  For every exponent $a\in(0,1)$, the function $t\mapsto t^a$ is concave and subadditive on $[0,\infty)$. In particular,
\[
(r+s)^a\leq r^a+s^a\quad \forall r,s\geq0.
\]
Applying this inequality with $r=\delta, \, a=1-\gamma(x)$, we obtain
\[
(s+\delta)^{1-\gamma(x)} - \delta^{1-\gamma(x)} \leq s^{1-\gamma(x)}.
\]
Consequently, by $\rm(H4)$, we have
\[
0\leq G_\delta(x,s) \leq \frac{s^{1-\gamma(x)}}{1-\gamma^+}
\]
for every $s\geq0$, almost every $x\in\Omega$, and every $0<\delta\leq1$. In particular, substituting $s=u(x)$ gives
\begin{equation}\label{G-a}
0\leq G_\delta(x,u) \leq \frac{u^{1-\gamma(x)}}{1-\gamma^+} \quad\text{for a.e. }x\in\Omega.
\end{equation}
Next, we decompose $\Omega$ into the measurable sets
\[
\Omega_0:=\{x\in\Omega:0\leq u(x)\leq1\} \quad \text{and} \quad \Omega_1:=\{x\in\Omega:u(x)>1\}.
\]
If $x\in\Omega_0$, then
\begin{equation}\label{omega-a}
u(x)^{1-\gamma(x)} \leq 1+u(x)^{p(x)}.
\end{equation}
If $x\in\Omega_1$, then
\begin{equation}\label{omega-b}
u(x)^{1-\gamma(x)} \leq u(x)^{p(x)} \leq 1+u(x)^{p(x)}.
\end{equation}
Combining \eqref{omega-a} and \eqref{omega-b}, we obtain the global pointwise inequality
\[
u(x)^{1-\gamma(x)} \leq 1+u(x)^{p(x)} \quad\text{for a.e. }x\in\Omega,
\]
from which one may write
\[
\begin{aligned}
\int_\Omega u(x)^{1-\gamma(x)}dx &= \int_{\Omega_0} u(x)^{1-\gamma(x)}dx + \int_{\Omega_1} u(x)^{1-\gamma(x)}dx \leq |\Omega_0| + \int_{\Omega_1}u(x)^{p(x)}dx \\ &\leq |\Omega| + \int_\Omega u(x)^{p(x)}dx.
\end{aligned}
\]
On the other hand, since $u\in\mathcal{M}_m^+$, we can write
\[
\rho_{p(\cdot)}(u) \leq p^+ \int_\Omega \frac{u(x)^{p(x)}}{p(x)} dx = p^+m.
\]
Consequently,
\begin{equation}\label{upest-aa}
\int_\Omega u(x)^{1-\gamma(x)}dx \leq |\Omega| + \int_\Omega u(x)^{p(x)}dx  \leq |\Omega|+p^+m.
\end{equation}
Thus the lower variable power $u^{1-\gamma(\cdot)}$ is integrable, and its integral is bounded uniformly over all $u\in\mathcal{M}_m^+$. Now using $\rm(H3)$ and the pointwise estimate \eqref{G-a}, we obtain
\[
\int_\Omega h(x)G_\delta(x,u(x))dx \leq \|h\|_{\infty} \int_\Omega G_\delta(x,u(x))dx \leq \frac{\|h\|_{\infty}}{1-\gamma^+} \int_\Omega u(x)^{1-\gamma(x)}dx.
\]
Lastly, applying the estimate \eqref{upest-aa} gives
\[
\int_\Omega h(x)G_\delta(x,u(x))dx \leq \frac{\|h\|_{\infty}}{1-\gamma^+} \left(|\Omega|+p^+m \right):=C_{sing}<\infty,
\]
for every $u\in\mathcal{M}_m^+$ and every $0<\delta\leq1$. This proves the lemma.
\end{proof}

\begin{corollary}\label{cor:coercive}
For $\mu$ in a bounded interval, the family $I_{\mu,\delta}$ is uniformly bounded below and equi-coercive on $\mathcal M_m^+$ for $0<\delta\leq1$.
\end{corollary}

\begin{proof}
Let $0\leq \mu\leq \mu_0$ and $0<\delta\leq1$, where $\mu_0<\infty$ is fixed. We prove that the family
\[
\left\{I_{\mu,\delta}: 0\leq\mu\leq\mu_0,\; 0<\delta\leq1 \right\}
\]
is uniformly bounded from below and equi-coercive on $\mathcal{M}_m^+$.\\
Multiplying \eqref{sing-integ} by $-\mu$ gives
\[
-\mu_0C_{\mathrm{sing}} \leq -\mu \int_\Omega h(x)G_\delta(x,u(x))dx \leq 0.
\]
Consequently,
\begin{equation}\label{funct-low-a}
I_{\mu,\delta}(u) = E_\varepsilon(u) - \mu \int_\Omega h(x)G_\delta(x,u(x))\,dx \geq E_\varepsilon(u) - \mu_0C_{\mathrm{sing}}.
\end{equation}
Using \eqref{lowbnd} gives
\begin{equation}\label{funct-low-b}
I_{\mu,\delta}(u) \geq E_\varepsilon(u) - \mu_0C_{\mathrm{sing}} \geq \frac{1}{p^+} \rho_{p(\cdot)}(\nabla u) - C_0,
\end{equation}
where $C_0 := C_1|\Omega| + \mu_0C_{\mathrm{sing}}$. From \eqref{funct-low-b}, one can obtain
\begin{equation}\label{funct-low-c}
\inf_{\substack{ 0\leq\mu\leq\mu_0\\ 0<\delta\leq1\\ u\in\mathcal{M}_m^+ }} I_{\mu,\delta}(u) \geq -C_0>-\infty.
\end{equation}
Thus the family $I_{\mu,\delta}$ is uniformly bounded from below on $\mathcal{M}_m^+$.\\
Next, we show the uniform boundedness of common sublevel sets. To do so, let $L\in\mathbb{R}$ be fixed, and consider the union of the $L$-sublevel sets
\[
\mathcal{S}_L:= \bigcup_{\substack{ 0\leq\mu\leq\mu_0\\ 0<\delta\leq1 }} \left\{u\in\mathcal{M}_m^+: I_{\mu,\delta}(u)\leq L \right\}.
\]
We show that $\mathcal{S}_L$ is bounded in $X$. Let $u\in\mathcal{S}_L$. By definition, there exist $\mu\in[0,\mu_0] \, \text{ and }\, \delta\in(0,1]$ such that
\[
I_{\mu,\delta}(u)\leq L.
\]
Using this and \eqref{funct-low-b}, we find
\begin{equation}\label{norm-bound}
0\leq \rho_{p(\cdot)}(\nabla u) \leq p^+(L+C_0):=K_L.
\end{equation}
If $\| u\|\leq1$, then the gradient norm is already bounded, thus assume that $\|u\|>1$. Then using \eqref{norm-bound} along with the norm-modular relations given in Proposition \ref{Prop:2.2}, we have
\[
\|u\|^{p^-} \leq \rho_{p(\cdot)}(\nabla u) \leq K_L,
\]
and hence
\begin{equation}
\|u\|\leq K_L^{1/p^-}.
\end{equation}
The right-hand side is independent of the particular choices of $u$, $\mu$, and $\delta$. Consequently, every common sublevel set $\mathcal{S}_L$ is bounded in $X$. Thus the family
\[
\left\{I_{\mu,\delta}: 0\leq\mu\leq\mu_0,\; 0<\delta\leq1 \right\}
\]
is equi-coercive on $\mathcal{M}_m^+$.
\end{proof}

\begin{lemma}\label{lem:constraint}
$J: L^{p(\cdot)}(\Omega)\rightarrow \mathbb{R}$, and $J':L^{p(\cdot)}(\Omega) \rightarrow \left(L^{p(\cdot)}(\Omega)\right)^* $ are continuous. Moreover,
\begin{equation}\label{deriv-J}
\langle J'(u),\phi\rangle=\int_\Omega |u|^{p(x)-2}u\phi dx,
\end{equation}
and $J'(u)\neq0$ on $\mathcal M_m$.
\end{lemma}

\begin{proof}
The continuity of $J$ and $J'$ and the formula \eqref{deriv-J} can be easily shown by using the standard arguments (see, e.g. \cite{FanZhang}). In particular, because the embedding $X \hookrightarrow L^{p(\cdot)}(\Omega)$ is continuous, the same results are also valid on $X$ and $X^{*}$. Let $u\in\mathcal{M}_m$. By definition, $J(u)=m>0$ if $u\neq0$. Therefore using the derivative formula \eqref{deriv-J} for $u\neq0$ gives
\[
\langle J'(u),u\rangle =\rho_{p(\cdot)}(u)\geq p^-m >0.
\]
Thus \(J'(u)\neq0\) on the prescribed-modular constraint$\mathcal{M}_m$.
\end{proof}

\begin{lemma}\label{thm:approx}
Under Assumption \ref{ass:base}, $I_{\mu,\delta}$ attains its minimum on $\mathcal M_m^+$ for every $\mu\geq0$ and $0<\delta\leq1$.
\end{lemma}

\begin{proof}
Fix $\mu\geq0 \, \text{ and }\,0<\delta\leq1$. Choose a function $\psi\in C_c^\infty(\Omega)$ such that $\psi\geq0 \,\text{ and }\, \psi\not\equiv0$. For $t\geq0$, define
\[
H(t):=J(t\psi) = \int_\Omega \frac{t^{p(x)}|\psi(x)|^{p(x)}}{p(x)} dx.
\]
We show that there exists $t_m>0$ such that $H(t_m)=m$. Then, $H$ is continuous on $[0,\infty)$. Indeed, let $t_n\to t$. Then, for almost every $x\in\Omega$,
\[
t_n^{p(x)}|\psi(x)|^{p(x)} \rightarrow t^{p(x)}|\psi(x)|^{p(x)}.
\]
Since $(t_n)$ is bounded, there exists $T>0$ such that $ 0\leq t_n\leq T$ for every sufficiently large $n$. Consequently,
\[
t_n^{p(x)} \leq \max\{T^{p^-},T^{p^+}\}
\]
for almost every \(x\in\Omega\). Therefore,
\[
\frac{t_n^{p(x)}|\psi(x)|^{p(x)}}{p(x)} \leq \frac{\max\{T^{p^-},T^{p^+}\}}{p^-} |\psi(x)|^{p(x)}.
\]
Because \(\psi\in C_c^\infty(\Omega)\), the function on the right-hand side belongs to $L^1(\Omega)$. Thus the DCT yields $H(t_n)\rightarrow H(t)$. Moreover, $H$ is strictly increasing on $(0,\infty)$. Indeed, if $0\leq t_1<t_2$, then $t_1^{p(x)}<t_2^{p(x)}$ for every $x$. Since the set $\{x\in\Omega:\psi(x)\neq0\}$ has positive measure, it follows that $ H(t_1)<H(t_2)$. Finally, $H(t)\to\infty$ as $t\to\infty$. To verify this, take $t\geq1$. Considering \eqref{p} and $\psi\not\equiv0$, we have
\[
H(t) = \int_\Omega \frac{t^{p(x)}|\psi(x)|^{p(x)}}{p(x)}dx \geq \frac{t^{p^-}}{p^+} \int_\Omega |\psi(x)|^{p(x)}dx,
\]
and therefore, $H(t)\rightarrow\infty \,\text{ as } t\rightarrow\infty.$
Since $H$ is continuous, and
\[
H(0)=0<m, \quad \lim_{t\to\infty}H(t)=\infty,
\]
the Intermediate Value Theorem gives a number $t_m>0$ such that $H(t_m)=m$. Then the function $v_m:=t_m\psi$ satisfies
\[
v_m\in X, \qquad v_m\geq0, \quad J(v_m)=m.
\]
Thus $v_m\in\mathcal{M}_m^+$, and hence $\mathcal{M}_m^+\neq\varnothing$.\\
Define
\[
c_{\varepsilon,m,\mu,\delta} := \inf_{u\in\mathcal{M}_m^+} I_{\mu,\delta}(u).
\]
Because \(\mathcal{M}_m^+\neq\varnothing\), the test function $v_m$ defined above gives
\[
c_{\varepsilon,m,\mu,\delta} \leq I_{\mu,\delta}(v_m).
\]
The quantity on the right is finite. Indeed, $E_\varepsilon(v_m)<\infty$ because $v_m\in X$, and Lemma~\ref{lem:primitive} gives
\[
0 \leq \int_\Omega h(x)G_\delta(x,v_m(x))dx \leq C_{\mathrm{sing}}<\infty.
\]
Therefore,
\[
c_{\varepsilon,m,\mu,\delta}<\infty.
\]
On the other hand, Corollary~\ref{cor:coercive}, applied with any fixed upper bound $\mu_0\geq\mu$, gives a constant $C_0>0$ such that
\[
I_{\mu,\delta}(u)\geq-C_0
\]
for every $u\in\mathcal{M}_m^+$. Hence
\[
c_{\varepsilon,m,\mu,\delta}\geq-C_0>-\infty.
\]
Consequently,
\[
-\infty < c_{\varepsilon,m,\mu,\delta} < \infty.
\]
By the definition of the infimum, there exists a sequence $(u_n)\subset\mathcal{M}_m^+$ such that
\[
I_{\mu,\delta}(u_n) \rightarrow c_{\varepsilon,m,\mu,\delta} \quad\text{as }n\to\infty.
\]
After discarding finitely many terms, we may assume that
\[
I_{\mu,\delta}(u_n) \leq c_{\varepsilon,m,\mu,\delta}+1
\]
for every large $n$. Let
\[
L:=c_{\varepsilon,m,\mu,\delta}+1.
\]
Then each \(u_n\) belongs to the $L$-sublevel set
\[
\left\{u\in\mathcal{M}_m^+: I_{\mu,\delta}(u)\leq L \right\}.
\]
By the equi-coercivity established in Corollary~\ref{cor:coercive}, this sublevel set is bounded in $X$. Thus there exists a constant $C_L>0$ such that
\[
\|u_n\| \leq C_L
\]
for every $n$. Indeed, using Lemma~\ref{lem:primitive}, $I_{\mu,\delta}(u_n)\leq L$, and  the lower growth estimate for $E_\varepsilon$ yields
\[
\rho_{p(\cdot)}(\nabla u_n) \leq p^+ \left(L+\mu C_{\mathrm{sing}}+C_1|\Omega| \right).
\]
The modular--norm inequalities then imply the boundedness of $(u_n)$ in $X$. Since $(u_n)$ is bounded in $X$, there exist a subsequence, not relabelled, and a function $u\in X$ such that
\begin{flalign}
&u_n\rightharpoonup u \text{ in } X, \label{eq:weaku} &\\
& u_n\rightarrow u \text{ in } L^{p(\cdot)}(\Omega),\label{eq:strongu}&\\
&\nabla u_n\rightharpoonup\nabla u \text{ in } L^{p(\cdot)}(\Omega),\label{eq:weakgrad}&\\
&u_n(x) \rightarrow u(x) \text{ for a.e. } x\in\Omega,\label{eq:aeu}.&
\end{flalign}
For every $n$, $u_n\in\mathcal{M}_m^+$, and hence, $u(x)\geq0$ for a.e. $x\in\Omega$ and $J(u_n)=m$. By Lemma~\ref{lem:constraint}, $ J(u_n)\rightarrow J(u)$, hence $J(u)=m.$ Together with
\[
u\in X \quad\text{and}\quad u\geq0\,\,\text{for a.e. }x\in\Omega,
\]
this proves that $u\in\mathcal{M}_m^+$. Thus the set $\mathcal{M}_m^+$ is closed.\\
Define
\begin{equation}\label{int-G}
\mathcal{G}_\delta(v):= \int_\Omega h(x)G_\delta(x,v)\,dx
\end{equation}
for nonnegative $v\in L^{p(\cdot)}(\Omega)$. We prove that $\mathcal{G}_\delta(u_n) \rightarrow \mathcal{G}_\delta(u)$. For almost every fixed $x\in\Omega$, the derivative of $G_\delta(x,s)$ is
\[
\frac{\partial G_\delta}{\partial s}(x,s) = (s+\delta)^{-\gamma(x)},\,\, s\geq0.
\]
Furthermore, because $0<\delta\leq1$ and $\gamma(x)\leq\gamma^+$,\quad $\delta^{-\gamma(x)} \leq \delta^{-\gamma^+}$.
Thus
\[
0< \frac{\partial G_\delta}{\partial s}(x,s) \leq \delta^{-\gamma^+}
\]
for all $s\geq0$ and almost every $x\in\Omega$. Applying the Mean Value Theorem, for any $a,b\geq0$, there exists a number $\zeta$ between $a$ and $b$ such that
\[
|G_\delta(x,a)-G_\delta(x,b)| = \left|\frac{\partial G_\delta}{\partial s}(x,\zeta) \right| |a-b| \leq \delta^{-\gamma^+}|a-b|.
\]
Therefore, using the Hölder inequality and the continuous embedding $L^{p(\cdot)}(\Omega) \hookrightarrow L^{1}(\Omega)$ gives
\[
\left|\mathcal{G}_\delta(u_n) - \mathcal{G}_\delta(u) \right| \leq \int_\Omega |h(x)| \left|G_\delta(x,u_n(x)) - G_\delta(x,u(x)) \right| dx \\
\leq 2\delta^{-\gamma^+} \|h\|_{\infty}  \|1\|_{p'(\cdot)}\|u_n-u\|_{p(\cdot)}
\]
which, consequently, implies that $\mathcal{G}_\delta(u_n) \rightarrow \mathcal{G}_\delta(u)$. Notice that the Lipschitz constant $\delta^{-\gamma^+}$ depends on the fixed approximation parameter $\delta$. This is sufficient here because the present theorem treats a fixed $\delta>0$. No estimate uniform as $\delta\downarrow0$ is required in this step.\\
By Lemma~\ref{lem:energy}, the regularized energy $E_\varepsilon$ is sequentially weakly lower semicontinuous. Since$ u_n\rightharpoonup u$ in $X$, we have
\[
E_\varepsilon(u) \leq \liminf_{n\to\infty} E_\varepsilon(u_n).
\]
Thus, considering also the continuity of $\mathcal{G}_\delta$ it reads
\begin{equation}\label{low-semi-cont}
I_{\mu,\delta}(u) \leq \liminf_{n\to\infty} \left(E_\varepsilon(u_n) - \mu\mathcal{G}_\delta(u_n) \right) = \liminf_{n\to\infty} I_{\mu,\delta}(u_n).
\end{equation}
Since $(u_n)$ is a minimizing sequence of $I_{\mu,\delta}$,
\[
I_{\mu,\delta}(u_n) \rightarrow c_{\varepsilon,m,\mu,\delta}.
\]
Therefore, using \eqref{low-semi-cont} it reads
\[
I_{\mu,\delta}(u) \leq c_{\varepsilon,m,\mu,\delta}.
\]
On the other hand, since $ u\in\mathcal{M}_m^+$, the definition of the infimum gives
\[
c_{\varepsilon,m,\mu,\delta} \leq I_{\mu,\delta}(u).
\]
Combining the last two inequalities yields
\[
I_{\mu,\delta}(u) = c_{\varepsilon,m,\mu,\delta}.
\]
Hence $u$ is a minimizer of $I_{\mu,\delta}$ on $\mathcal{M}_m^+$; that is, for every $\mu\geq0$ and $0<\delta\leq1$, there exists $ u_{\varepsilon,m,\mu,\delta} \in\mathcal{M}_m^+$ such that
\[
I_{\mu,\delta} \left(u_{\varepsilon,m,\mu,\delta} \right) = \inf_{v\in\mathcal{M}_m^+} I_{\mu,\delta}(v).
\]
This completes the proof.
\end{proof}

In Lemma \ref{prop:euler} below, we construct the approximate  normalized weak solution pairs $(u_\delta,\lambda_\delta) \in \mathcal M_m^+ \times \mathbb{R}$ used in Theorem \ref{thm:limit}.
\begin{lemma}\label{prop:euler}
Let $u_\delta$ minimize $I_{\mu,\delta}$ on $\mathcal M_m^+$. Suppose signed constrained variations are admissible at $u_\delta$. Then there is a unique Lagrange multiplier $\lambda_\delta\in\mathbb{R}$ such that the Euler–Lagrange identity
\begin{equation}\label{Euler–Lagrange-identity}
\int_\Omega A_\varepsilon(x,\nabla u_\delta)\cdot\nabla\phi dx=\lambda_\delta\int_\Omega u_\delta^{p(x)-1}\phi dx+\mu\int_\Omega h(u_\delta+\delta)^{-\gamma(x)}\phi dx
\end{equation}
holds for every $\phi\in X$.
\end{lemma}

\begin{proof}
Let let $u_\delta = u_{\varepsilon,m,\mu,\delta} \in\mathcal{M}_m^+$ be a minimizer of $I_{\mu,\delta}$ on $\mathcal{M}_m^+$. Thus
\[
u_\delta\geq0 \quad\text{a.e. in }\Omega, \quad J(u_\delta)=m,
\]
and
\[
I_{\mu,\delta}(u_\delta) \leq I_{\mu,\delta}(v) \quad \text{for every }v\in\mathcal{M}_m^+.
\]
By assumption, signed constrained variations are admissible at $u_\delta$. More precisely, the positivity and truncation-renormalization argument ensures that $u_\delta$ is a critical point of $I_{\mu,\delta}$ with respect to all variations tangent to the equality constraint $J(u)=m$.\\
By Lemma \ref{lem:energy},
\[
\left\langle E_\varepsilon'(u_\delta),\phi \right\rangle = \int_\Omega A_\varepsilon(x,\nabla u_\delta) \cdot\nabla\phi dx\quad  \forall \phi\in X.
\]
The integral on the right is finite. Indeed, from Lemma \ref{lem:energy} we have $A_\varepsilon(\cdot,\nabla u(\cdot)) \in L^{p'(\cdot)}(\Omega)$. Thus, the Hölder inequality yields
\begin{equation}\label{lem2.8-aa}
\left|\int_\Omega A_\varepsilon(x,\nabla u_\delta) \cdot\nabla\phi dx \right| \leq 2 \left\|A_\varepsilon(\cdot,\nabla u_\delta)) \right\|_{p'(\cdot)} \|\phi\|.
\end{equation}
We now compute the directional derivative of $\mathcal{G}_\delta$ at $u_\delta$. Let $\phi\in X$ be an admissible signed direction, and let $t$ be sufficiently small so that the corresponding truncated or local variation is defined. The Fundamental Theorem of Calculus gives, for almost every $x\in\Omega$,
\[
G_\delta \left(x,u_\delta(x)+t\phi(x)\right) - G_\delta \left(x,u_\delta(x)\right) = t \int_0^1 \left(u_\delta(x)+st\phi(x)+\delta \right)^{-\gamma(x)} \phi(x)ds.
\]
Hence
\[
\frac{ \mathcal{G}_\delta(u_\delta+t\phi) - \mathcal{G}_\delta(u_\delta) }{t} = \int_\Omega h(x) \int_0^1 \left(u_\delta(x)+st\phi(x)+\delta \right)^{-\gamma(x)} \phi(x)ds\,dx.
\]
Since $0<\delta\leq1$ and $\gamma(x)\leq\gamma^+$, we have $0< (s+\delta)^{-\gamma(x)} \leq \delta^{-\gamma^+}$. Thus
\[
\left|h(x) \left(u_\delta(x)+st\phi(x)+\delta \right)^{-\gamma(x)} \phi(x) \right| \leq \delta^{-\gamma^+} \|h\|_{\infty} |\phi(x)|.
\]
By the continuous embedding $L^{p(\cdot)}(\Omega)\hookrightarrow L^1(\Omega)$, $\phi\in L^1(\Omega)$, and hence, the right-hand side is integrable. Additionally, for almost every $x\in\Omega$ and every $s\in[0,1]$,
\[
\left(u_\delta(x)+st\phi(x)+\delta \right)^{-\gamma(x)} \rightarrow \left(u_\delta(x)+\delta \right)^{-\gamma(x)}
\]
as $t\to0$. Thus the DCT gives
\[
\left\langle \mathcal{G}_\delta'(u_\delta),\phi \right\rangle = \int_\Omega h(x) \left(u_\delta(x)+\delta \right)^{-\gamma(x)} \phi(x) dx.
\]
This expression, i.e. Gâteaux derivative of $\mathcal{G}_\delta$, defines a  linear (in $\phi$) and bounded functional on $X$ since
\[
\left|\left\langle \mathcal{G}_\delta'(u_\delta),\phi \right\rangle \right| \leq C\delta^{-\gamma^+} \|h\|_{\infty}  \|1\|_{p'(\cdot)} \|\phi\|.
\]
Since
\[
I_{\mu,\delta}(u) = E_\varepsilon(u) - \mu\mathcal{G}_\delta(u),
\]
the preceding calculations give
\[
\left\langle I_{\mu,\delta}'(u_\delta),\phi \right\rangle = \left\langle E_\varepsilon'(u_\delta),\phi \right\rangle - \mu \left\langle \mathcal{G}_\delta'(u_\delta),\phi \right\rangle = \int_\Omega A_\varepsilon(x,\nabla u_\delta) \cdot\nabla\phi dx - \mu \int_\Omega h(x) \left(u_\delta+\delta \right)^{-\gamma(x)} \phi dx.
\]
By Lemma~\ref{lem:constraint}, the functional $J$ is continuously differentiable, and for $u_\delta\geq0$, it reads
\[
\left\langle J'(u_\delta),\phi \right\rangle = \int_\Omega u_\delta(x)^{p(x)-1}\phi(x) dx.
\]
The derivative  $J'(u_\delta)$ is not the zero functional. Indeed, testing it in the direction $u_\delta$ gives
\[
\langle J'(u_\delta),u_\delta\rangle =\rho_{p(\cdot)}(u_\delta)\geq p^-m >0,
\]
and hence, $J'(u_\delta)\neq0$. It follows that the level set
\[
\mathcal{M}_m = \{u\in X:J(u)=m\}
\]
is locally a $C^1$ codimension-one constraint surface near $u_\delta$. Its tangent space at $u_\delta$ is
\[
T_{u_\delta}\mathcal{M}_m = \ker J'(u_\delta) = \left\{\phi\in X: \int_\Omega u_\delta^{p(x)-1}\phi dx=0\right\}.
\]
Next, we need to show that the first variation vanishes in every tangent direction. Let $\phi\in\ker J'(u_\delta)$. Thus $\left\langle J'(u_\delta),\phi \right\rangle = 0$. Since $J'(u_\delta)\neq0$, choose $\psi\in X$ such that $\left\langle J'(u_\delta),\psi \right\rangle \neq0$. For example, one may take $\psi=u_\delta$, because
\[
\left\langle J'(u_\delta),u_\delta \right\rangle = \rho_{p(\cdot)}(u_\delta)>0.
\]
Define
\[
F(t,s):= J(u_\delta+t\phi+s\psi)-m.
\]
Then
\[
F(0,0)=J(u_\delta)-m=0 \quad  \text{and}\quad \frac{\partial F}{\partial s}(0,0) = \left\langle J'(u_\delta),\psi \right\rangle \neq0.
\]
By the implicit function theorem, there exist $\tau>0$ and a continuously differentiable function $s:(-\tau,\tau)\longrightarrow\mathbb{R}$ such that
\[
s(0)=0  \quad  \text{and}\quad F(t,s(t))=0 \quad\text{for all }|t|<\tau.
\]
Equivalently,
\[
J\left(u_\delta+t\phi+s(t)\psi \right) = m.
\]
Thus the curve
\[
v(t):= u_\delta+t\phi+s(t)\psi
\]
lies in the equality-constraint surface $\mathcal{M}_m$. Differentiating the identity $J(v(t))=m$ at $t=0$, and considering that
\[
v'(0)=\phi+s'(0)\psi,
\]
gives
\[
0 = \left. \frac{d}{dt}J(v(t)) \right|_{t=0} = \left\langle J'(u_\delta), \phi+s'(0)\psi \right\rangle = \left\langle J'(u_\delta),\phi \right\rangle + s'(0) \left\langle J'(u_\delta),\psi \right\rangle.
\]
The first term is zero because $\phi\in\ker J'(u_\delta)$. Hence $s'(0) \left\langle J'(u_\delta),\psi \right\rangle = 0$.
However, since $\left\langle J'(u_\delta),\psi \right\rangle \neq0$, we conclude that $ s'(0)=0$, and therefore, $v'(0)=\phi$. By the signed-variation assumption, the curve $v(t)$, or its truncated and renormalized approximation, is admissible in $\mathcal{M}_m^+$ for both positive and negative sufficiently small $t$. Additionally, since $u_\delta=v(0)$ is a minimizer,
\[
I_{\mu,\delta}(v(t)) \geq I_{\mu,\delta}(v(0))
\]
for all sufficiently small $t$, with either sign. Thus the real-valued function $t\mapsto I_{\mu,\delta}(v(t))$ has a local minimum at $t=0$ yielding
\[
0 = \left. \frac{d}{dt} I_{\mu,\delta}(v(t)) \right|_{t=0}= \left\langle I_{\mu,\delta}'(u_\delta),v'(0) \right\rangle = \left\langle I_{\mu,\delta}'(u_\delta),\phi \right\rangle.
\]
Therefore
\[
\left\langle I_{\mu,\delta}'(u_\delta),\phi \right\rangle = 0 \quad \text{for every } \phi\in\ker J'(u_\delta)
\]
which, equivalently means, that
\begin{equation}\label{kernel}
\ker J'(u_\delta) \subset \ker I_{\mu,\delta}'(u_\delta).
\end{equation}
We now derive the Lagrange multiplier directly. Choose $\psi\in X$ such that $\left\langle J'(u_\delta),\psi \right\rangle \neq0$.
Define
\[
\lambda_\delta:= \frac{ \left\langle I_{\mu,\delta}'(u_\delta),\psi \right\rangle }{\left\langle J'(u_\delta),\psi \right\rangle }.
\]
We hall show that
\[
I_{\mu,\delta}'(u_\delta) = \lambda_\delta J'(u_\delta) \quad\text{in }X^*.
\]
Let \(\phi\in X\) be arbitrary, and define
\[
\phi_0:= \phi - \frac{ \left\langle J'(u_\delta),\phi \right\rangle }{ \left\langle J'(u_\delta),\psi \right\rangle } \psi.
\]
Then
\[
\left\langle J'(u_\delta),\phi_0 \right\rangle = \left\langle J'(u_\delta),\phi \right\rangle - \frac{\left\langle J'(u_\delta),\phi \right\rangle }{\left\langle J'(u_\delta),\psi \right\rangle } \left\langle J'(u_\delta),\psi \right\rangle = 0.
\]
Therefore, $\phi_0\in\ker J'(u_\delta)$. This, by \eqref{kernel}, implies
\[
\left\langle I_{\mu,\delta}'(u_\delta),\phi_0 \right\rangle = 0.
\]
Expanding this identity gives
\[
0 = \left\langle I_{\mu,\delta}'(u_\delta),\phi \right\rangle - \frac{\left\langle J'(u_\delta),\phi \right\rangle }{\left\langle J'(u_\delta),\psi \right\rangle } \left\langle I_{\mu,\delta}'(u_\delta),\psi \right\rangle.
\]
By the definition of $\lambda_\delta$,
\begin{equation}\label{lag-multip-a}
\left\langle I_{\mu,\delta}'(u_\delta),\phi \right\rangle = \lambda_\delta \left\langle J'(u_\delta),\phi \right\rangle
\end{equation}
for every $\phi\in X$. Thus
\begin{equation}\label{lag-multip-b}
I_{\mu,\delta}'(u_\delta) = \lambda_\delta J'(u_\delta).
\end{equation}
A particularly convenient choice is $\psi=u_\delta$. In that case,
\[
\left\langle J'(u_\delta),u_\delta \right\rangle = \int_\Omega u_\delta^{p(x)}dx,
\]
and therefore
\begin{equation}\label{lag-multip-c}
\lambda_\delta = \frac{\left\langle I_{\mu,\delta}'(u_\delta),u_\delta \right\rangle }{ \displaystyle \int_\Omega u_\delta^{p(x)}dx  },
\end{equation}
with noting that the denominator is strictly positive and satisfies
\[
p^-m \leq \int_\Omega u_\delta^{p(x)}dx \leq p^+m.
\]
Using the derivative formula for $I_{\mu,\delta}$, \eqref{lag-multip-c} provides \eqref{Euler–Lagrange-identity}.
\[
\int_\Omega A_\varepsilon(x,\nabla u_\delta) \cdot\nabla\phi dx = \lambda_\delta \int_\Omega u_\delta^{p(x)-1}\phi dx + \mu \int_\Omega h(x) \left(u_\delta+\delta\right)^{-\gamma(x)} \phi dx
\]
To show the uniqueness of the Lagrange multiplier, suppose that two numbers $\lambda_\delta^{(1)}$ and $\lambda_\delta^{(2)}$ both satisfy the Euler-Lagrange identity \eqref{Euler–Lagrange-identity}. Subtracting the two identities gives
\[
\left(\lambda_\delta^{(1)} - \lambda_\delta^{(2)} \right) \int_\Omega u_\delta^{p(x)-1}\phi dx = 0\quad \forall \phi\in X.
\]
Choose $\phi=u_\delta$.  Then
\[
\left(\lambda_\delta^{(1)} - \lambda_\delta^{(2)} \right) \int_\Omega u_\delta^{p(x)}dx = 0.
\]
Since
\[
\int_\Omega u_\delta^{p(x)} dx \geq p^-m>0,
\]
we conclude that $\lambda_\delta^{(1)} = \lambda_\delta^{(2)}$. The proof is complete.
\end{proof}

\section{Passage to the Singular Limit} \label{sec:limit}

\begin{lemma}\label{lem:duality}
Suppose Assumptions~\ref{ass:base} and~\ref{ass:barrier} hold. Then $f_\delta(\cdot)=h(\cdot)(u_\delta(\cdot)+\delta)^{-\gamma(\cdot)}$ is uniformly bounded in $X^*$. If $\delta_n\downarrow0$ and $u_{\delta_n}\to u$ a.e., then $u \geq c_*d^\vartheta>0$ a.e. and
\begin{equation}\label{lem3.1-aa}
\int_\Omega h(x)(u_{\delta_n}+\delta_n)^{-\gamma(x)}\phi dx \to\int_\Omega h(x)u^{-\gamma(x)}\phi dx
\end{equation}
for every $\phi\in X$.
\end{lemma}

\begin{proof}
For almost every \(x\in\Omega\) and every $0<\delta\leq\delta_0$, $(u_\delta(x)+\delta)^{-\gamma(x)} \leq u_\delta(x)^{-\gamma(x)}$. Therefore, using \rm(H5), we obtain
\[
u_\delta(x)^{-\gamma(x)} \leq \left(c_*d(x)^\vartheta\right)^{-\gamma(x)} = c_*^{-\gamma(x)} d(x)^{-\vartheta\gamma(x)}.
\]
Since the map $q\mapsto c_*^{-q}$ is continuous on the compact interval $[\gamma^-,\gamma^+]$, its maximum occurs at one of the endpoints. Define
\[
C_c := \max\left\{ c_*^{-\gamma^-}, c_*^{-\gamma^+} \right\}.
\]
Then
\[
c_*^{-\gamma(x)} \leq C_c
\]
for every $x\in\Omega$. Consequently,
\begin{equation}\label{lem3.1-a}
\left(u_\delta(x)+\delta\right)^{-\gamma(x)} \leq u_\delta(x)^{-\gamma(x)} \leq C_c d(x)^{-\vartheta\gamma(x)}
\end{equation}
for almost every $x\in\Omega$, and every $0<\delta\leq\delta_0$. Multiplying \eqref{lem3.1-a} by $h(x)\geq0$, we obtain
\begin{equation}\label{lem3.1-b}
0\leq f_\delta(x) \leq C_c h(x)d(x)^{-\vartheta\gamma(x)}.
\end{equation}
Let $\phi\in X$ be arbitrary. Then using \rm(H6) gives
\begin{equation}\label{lem3.1-c}
|f_\delta(x)\phi(x)| \leq C_c h(x) d(x)^{-\vartheta\gamma(x)} |\phi(x)| = C_ch(x) d(x)^{1-\vartheta\gamma(x)} \frac{|\phi(x)|}{d(x)}\leq C_c \omega_\vartheta(x) \frac{|\phi(x)|}{d(x)}.
\end{equation}
Integrating \eqref{lem3.1-c} and applying the Hölder inequality along with Lemma~\ref{lem:hardy} gives
\begin{equation}\label{lem3.1-d}
 \left|\int_\Omega f_\delta(x)\phi(x)dx \right| \leq C_* \|\phi\|
\end{equation}
for every $\phi\in X$ and every $0<\delta\leq\delta_0$, where $ C_* := 2C_cC_H \|\omega_\vartheta\|_{p'(\cdot)}$. Notice that the constant $C_*$ depends on $ c_*, \,\gamma^-, \, \gamma^+, \, C_H$ and $\|\omega_\vartheta\|_{p'(\cdot)}$, but it is independent of $\delta$.\\
For each $0<\delta\leq\delta_0$, define
\[
\mathcal{F}_\delta: X \longrightarrow \mathbb{R}
\]
by
\[
\left\langle \mathcal{F}_\delta,\phi \right\rangle:= \int_\Omega f_\delta(x)\phi(x)dx.
\]
The map $\mathcal{F}_\delta$ is linear in $\phi$. Then
\[
\left|\left\langle \mathcal{F}_\delta,\phi \right\rangle \right| \leq C_* \|\phi\|.
\]
Therefore, $\mathcal{F}_\delta \in X^*$ and
\[
\|\mathcal{F}_\delta\|_{*} = \sup_{\substack{ \phi\in X\\ \|\phi\|\leq 1 }} \left|\left\langle \mathcal{F}_\delta,\phi \right\rangle \right| \leq C_*.
\]
Hence
\begin{equation}\label{lem3.1-def}
\sup_{0<\delta\leq\delta_0} \|f_\delta\|_{*} \leq C_*<\infty .
\end{equation}
This proves the uniform boundedness of the regularized singular sources in the dual space.\\
Suppose now that
\[
\delta_n\downarrow0 \quad \text{and} \quad u_{\delta_n}(x)\rightarrow u(x)
\]
for almost every $x\in\Omega$ with $0<\delta_n\leq \delta_0$. Then by \rm(H5)
\[
u_{\delta_n}(x) \geq c_*d(x)^\vartheta>0
\]
for every $n$, and for almost every $x\in\Omega$. Therefore,
\[
 u(x)>0, \, \text{and hence, }\, u(x)^{-\gamma(x)}>0 \quad\text{for almost every }x\in\Omega.
\]
Since $ u_{\delta_n}(x)\rightarrow u(x)$ for almost every $x\in\Omega$,  and $\delta_n\rightarrow0$,
\[
u_{\delta_n}(x)+\delta_n \rightarrow u(x)
\]
from which it follows
\[
\left(u_{\delta_n}(x)+\delta_n\right)^{-\gamma(x)} \rightarrow u(x)^{-\gamma(x)}
\]
for almost every $x\in\Omega$. Multiplication by the fixed measurable coefficient $h(x)$ gives
\[
h(x)\left(u_{\delta_n}(x)+\delta_n\right)^{-\gamma(x)} \rightarrow h(x)u(x)^{-\gamma(x)}
\]
for almost every $x\in\Omega$. Thus, for every fixed  test function $\phi$,
\begin{equation}\label{lem3.1-e}
h(x) \left(u_{\delta_n}(x)+\delta_n\right)^{-\gamma(x)} \phi(x) \rightarrow h(x)u(x)^{-\gamma(x)}\phi(x)
\end{equation}
for almost every $x\in\Omega$. Fix $\phi\in X$. By \eqref{lem3.1-a},
\begin{equation}\label{lem3.1-f}
\left|h(x) \left(u_{\delta_n}(x)+\delta_n\right)^{-\gamma(x)} \phi(x) \right| \leq C_c h(x)d(x)^{-\vartheta\gamma(x)} |\phi(x)| = C_c \omega_\vartheta(x) \frac{|\phi(x)|}{d(x)}:=H_\phi(x).
\end{equation}
Considering Assumption~\ref{ass:barrier}, and applying the Hölder inequality and Lemma \ref{lem:hardy} gives
\begin{equation}\label{lem3.1-ff}
\int_\Omega H_\phi(x) dx = C_c \int_\Omega \omega_\vartheta(x) \frac{|\phi(x)|}{d(x)} dx \leq 2C_cC_H \|\omega_\vartheta\|_{p'(\cdot)}\|\phi\|= C_* \|\phi\|<\infty.
\end{equation}
Notice that the function $H_\phi$ is independent of $n$, and the same estimate also shows that the limiting singular term is globally integrable against $\phi$. Indeed, by \eqref{lem3.1-a}
\[
h(x)u(x)^{-\gamma(x)}|\phi(x)| \leq H_\phi(x),
\]
and hence, $h u^{-\gamma(\cdot)}\phi \in L^1(\Omega)$. Finally, by \eqref{lem3.1-e}-\eqref{lem3.1-ff}, the DCT provides \eqref{lem3.1-aa}. This completes the proof; however, we would also like to note that the limiting source $f(x):=h(x)u(x)^{-\gamma(x)}$ belongs to $X^*$.  Indeed, using the boundary estimate \eqref{lem3.1-a}, and repeating the calculations from above, one can obtain
\[
\left|\int_\Omega h(x)u(x)^{-\gamma(x)}\phi(x)dx \right| \leq C_c \int_\Omega \omega_\vartheta(x) \frac{|\phi(x)|}{d(x)} dx = \int_\Omega H_\phi(x)dx \leq C_* \|\phi\|.
\]
Thus
\begin{equation}\label{lem3.1-ffg}
f\in X^* \quad \text{and} \quad \|f\|_ {*} \leq C_*.
\end{equation}
\end{proof}

\begin{lemma}\label{lem:splus}
The operator $\mathcal{A}_\varepsilon :X \longrightarrow X^*$ defined by
\[
\langle \mathcal A_\varepsilon u,\phi\rangle=\int_\Omega A_\varepsilon(x,\nabla u)\cdot\nabla\phi dx
\]
is strictly monotone, and of type $(S_+)$; that is,\\
If $u_n\rightharpoonup u$ in $X$ and $\limsup_n\langle\mathcal A_\varepsilon u_n,u_n-u\rangle\leq0$, then $u_n\to u$ in $X$.
\end{lemma}

\begin{proof}
For $u, \phi \in X$, since $A_\varepsilon(\cdot,\eta(\cdot)) \in L^{p'(\cdot)}(\Omega)$, it reads
\[
\left|\left\langle \mathcal{A}_\varepsilon u,\phi \right\rangle \right|=\left|\int_\Omega A_\varepsilon(x,\nabla u) \cdot\nabla\phi dx\right| \leq 2 \left\|A_\varepsilon(\cdot,\nabla u) \right\|_{p'(\cdot)} \|\phi\|,
\]
which implies that $\mathcal{A}_\varepsilon u\in X^*$. Moreover, if $(u_n)$ is bounded in $X$, then $(\nabla u_n)$ is bounded in $L^{p(\cdot)}(\Omega)$. Additionally, by Lemma \ref{lem:energy}, the Nemytskii operator  $\eta \mapsto A_\varepsilon(\cdot,\eta(\cdot))$ maps $L^{p(\cdot)}(\Omega)$ boundedly into $L^{p'(\cdot)}(\Omega)$, hence, the sequence $(A_\varepsilon(x,\nabla u_n))$ is bounded in $L^{p'(\cdot)}(\Omega)$. Thus \(\mathcal{A}_\varepsilon\) maps bounded subsets of $X$ into bounded subsets of $X^*$.\\
By Lemma \ref{lem:quantitative-monotonicity-flux}, $A_\varepsilon(x,\cdot)$ is strictly convex. Let $u,v\in X$. Then
\begin{equation}\label{lema3.2-b}
\left\langle \mathcal{A}_\varepsilon u - \mathcal{A}_\varepsilon v, u-v \right\rangle = \int_\Omega \left(A_\varepsilon(x,\nabla u)-A_\varepsilon(x,\nabla v) \right) \cdot (\nabla u-\nabla v)dx \geq 0.
\end{equation}
Suppose equality holds. Since the integrand is nonnegative, its integral can vanish only if
\[
\left(A_\varepsilon(x,\nabla u(x)) -A_\varepsilon(x,\nabla v(x)) \right) \cdot \left(\nabla u(x)-\nabla v(x) \right) = 0
\]
a.e. in $\Omega$. Pointwise strict monotonicity of $A_\varepsilon(x,\cdot)$ then implies $\nabla u(x)=\nabla v(x)$ a.e. in $\Omega$. Thus $\nabla(u-v)=0$ a.e. in $\Omega$.  Because $u-v\in X$,
\[
\|u-v\| = 0.
\]
Hence $u=v$ a.e. in $\Omega$. This proves the strict monotonicity of $\mathcal{A}_\varepsilon$; that is,
\begin{equation}\label{lema3.2-c}
\left\langle \mathcal{A}_\varepsilon u - \mathcal{A}_\varepsilon v, u-v \right\rangle >0 \quad\text{whenever }u\neq v.
\end{equation}
Suppose $u_n \in X$ for every $n$ with $u_n\rightharpoonup u$ in $X$. Because $u\in X$, $\mathcal{A}_\varepsilon u\in X^*$. Thus the weak convergence of $(u_n)$ to $u$ gives
\[
(u_n-u)\rightharpoonup0 \quad\text{in }X, \, \text{ and therefore}, \quad \left\langle \mathcal{A}_\varepsilon u, u_n-u \right\rangle \rightarrow0.
\]
Define
\[
D_n:= \left\langle \mathcal{A}_\varepsilon u_n - \mathcal{A}_\varepsilon u, u_n-u \right\rangle = \int_\Omega \left(A_\varepsilon(x,\nabla u_n) - A_\varepsilon(x,\nabla u) \right) \cdot (\nabla u_n-\nabla u)\,dx.
\]
By monotonicity, $D_n\geq0$. Moreover, since
\[
\begin{aligned}
D_n &= \left\langle \mathcal{A}_\varepsilon u_n, u_n-u \right\rangle - \left\langle \mathcal{A}_\varepsilon u, u_n-u \right\rangle,
\end{aligned}
\]
hence
\[
\limsup_{n\to\infty}D_n \leq \limsup_{n\to\infty} \left\langle \mathcal{A}_\varepsilon u_n, u_n-u \right\rangle - \lim_{n\to\infty} \left\langle \mathcal{A}_\varepsilon u, u_n-u \right\rangle \leq 0.
\]
Since $D_n\geq0$, it follows that $D_n\rightarrow0$. Therefore
\begin{equation}\label{lema3.2-d}
\int_\Omega \left(A_\varepsilon(x,\nabla u_n)-A_\varepsilon(x,\nabla u) \right) \cdot (\nabla u_n-\nabla u) dx \rightarrow 0.
\end{equation}
Next, we partition $\Omega$ into the measurable sets
\[
\Omega_{\geq2}:= \{x\in\Omega:p(x)\geq2\} \quad \text{and} \quad \Omega_{<2}:= \{x\in\Omega:1<p(x)<2\}.
\]
For $x \in \Omega_{\geq2}$: \eqref{lema3.2-e} gives
\[
c_0 \int_{\Omega_{\geq2}} |w_n|^{p(x)}dx \leq \int_{\Omega_{\geq2}} (A_\varepsilon(x,\nabla u_n)-A_\varepsilon(x,\nabla u)) \cdot w_ndx \leq D_n.
\]
where $w_n(x):=\nabla u_n(x)-\nabla u(x)$. Since $D_n\to 0$, it follows that
\begin{equation} \label{lema3.2-g}
\int_{\Omega_{\geq2}} |\nabla u_n-\nabla u|^{p(x)}dx \rightarrow 0.
\end{equation}
For $x\in\Omega_{<2}$: define
\[
S_n(x):= \varepsilon+ |\nabla u_n(x)|+ |\nabla u(x)|.
\]
Then \eqref{lema3.2-f} gives
\[
c_0 \int_{\Omega_{<2}} |w_n|^2S_n^{p(x)-2}dx \leq \int_{\Omega_{<2}}  \left(A_\varepsilon(x,\nabla u_n)-A_\varepsilon(x,\nabla u) \right) \cdot w_n dx \leq D_n
\]
Thus
\[
\int_{\Omega_{<2}} |w_n|^2S_n^{p(x)-2}dx \rightarrow 0.
\]
Set
\[
g_n(x):= |w_n(x)|^2S_n(x)^{p(x)-2}.
\]
Then $g_n\geq0$ and
\[
\int_{\Omega_{<2}}g_n(x)dx\rightarrow0.
\]
On $\Omega_{<2}$, we can write the factorization
\[
|w_n|^{p(x)} = \left(|w_n|^2S_n^{p(x)-2} \right)^{p(x)/2} S_n^{\frac{p(x)(2-p(x))}{2}} = g_n^{p(x)/2} S_n^{\frac{p(x)(2-p(x))}{2}}.
\]
Applying the Hölder inequality for the conjugate exponents $r(x):=\frac{2}{p(x)}$ and $ r'(x):=\frac{2}{2-p(x)}$  gives
\[
\int_{\Omega_{<2}}|w_n|^{p(x)}dx \leq 2 \left\| g_n^{p(\cdot)/2} \right\|_{L^{r(\cdot)}(\Omega_{<2})} \times \left\| S_n^{\frac{p(\cdot)(2-p(\cdot))}{2}} \right\|_{L^{r'(\cdot)}(\Omega_{<2})}.
\]
We estimate the two factors separately. For the first factor,
\[
\rho_{r(\cdot)} \left(g_n^{p(\cdot)/2} \right) = \int_{\Omega_{<2}} \left(g_n^{p(x)/2} \right)^{2/p(x)} dx =\int_{\Omega_{<2}}g_n dx \rightarrow 0.
\]
The modular--norm relations therefore imply
\begin{equation}\label{lema3.2-h}
\left\|g_n^{p(\cdot)/2} \right\|_{L^{r(\cdot)}(\Omega_{<2})} \rightarrow 0.
\end{equation}
For the second factor, similarly
\[
\rho_{r'(\cdot)} \left(S_n^{\frac{p(\cdot)(2-p(\cdot))}{2} }\right)  = \int_{\Omega_{<2}}S_n^{p(x)} dx.
\]
Since
\[
S_n^{p(x)} \leq C \left(\varepsilon^{p(x)} + |\nabla u_n|^{p(x)} + |\nabla u|^{p(x)}\right),
\]
where $C$ depends only on $p^+$, it follows
\[
\int_{\Omega_{<2}}S_n^{p(x)} dx \leq C\max\{ \varepsilon^{p^-}, \varepsilon^{p^+} \}|\Omega| + C \int_\Omega (|\nabla u_n|^{p(x)}+ |\nabla u|^{p(x)})dx.
\]
Since $u_n\rightharpoonup u$ in $X$, it reads
\[
\sup_n \int_{\Omega_{<2}}S_n^{p(x)}dx <\infty.
\]
The modular--norm relations imply
\begin{equation}\label{lema3.2-j}
\sup_n \left\|S_n^{ \frac{p(\cdot)(2-p(\cdot))}{2} } \right\|_{L^{r'(\cdot)}(\Omega_{<2})} <\infty.
\end{equation}
Combining \eqref{lema3.2-h} and \eqref{lema3.2-j} gives
\begin{equation}\label{lema3.2-k}
\int_{\Omega_{<2}} |\nabla u_n-\nabla u|^{p(x)} dx \rightarrow 0.
\end{equation}
Combining \eqref{lema3.2-g} and \eqref{lema3.2-k} gives
\[
\rho_{p(\cdot)} \left(\nabla u_n-\nabla u \right) = \int_\Omega |\nabla u_n-\nabla u|^{p(x)} dx  \rightarrow 0,
\]
which implies
\begin{equation}\label{lema3.2-m}
\| u_n-u\| \rightarrow0.
\end{equation}
This completes the proof.
\end{proof}

Theorem \ref{thm:limit}, the main result of the paper, removes the singular regularization $\delta$ of \eqref{eq:main-eigen-system}, preserves normalization and positivity, and identifies the limiting pair as a weak solution of  \eqref{eq:main-eigen-system} at fixed $\varepsilon>0$.

\begin{theorem} \label{thm:limit} Suppose Assumptions~\ref{ass:base} and~\ref{ass:barrier} hold. Fix $\varepsilon>0, \, m>0, \,\mu\geq0$. Let $0<\delta_n\leq1$ with $\delta_n\downarrow0$ as $n \to \infty$. For each $n$, let
\[
\left(u_n,\lambda_n \right): = \left(u_{\varepsilon,m,\mu,\delta_n}, \lambda_{\varepsilon,m,\mu,\delta_n} \right) \in X\times\mathbb R
\]
be a normalized weak solution of the $(\varepsilon,\delta_n)$-regularized approximate problem
\begin{equation}\label{eq:approx-euler-limit-proof}
\int_\Omega A_\varepsilon(x,\nabla u_n)\cdot\nabla\phi dx = \lambda_n \int_\Omega u_n^{p(x)-1}\phi dx + \mu \int_\Omega h(x)(u_n+\delta_n)^{-\gamma(x)} \phi dx, \quad  J(u_n)=m,
\end{equation}
for every $\phi\in X$.\\
Then, after passing to a subsequence, there exist $u_{\varepsilon,m,\mu} \in X$ and $\lambda_{\varepsilon,m,\mu} \in\mathbb R$ such that
\[
u_{\varepsilon,m,\mu,\delta_n} \rightarrow u_{\varepsilon,m,\mu} \, \text{ in } X \quad \text{and} \quad \lambda_{\varepsilon,m,\mu,\delta_n} \rightarrow \lambda_{\varepsilon,m,\mu} \, \text{ in }\mathbb R.
\]
Moreover,
\[
J(u_{\varepsilon,m,\mu})=m\quad \text{and} \quad u_{\varepsilon,m,\mu} \geq c_*d^\vartheta > 0 \, \text{almost everywhere in }\Omega.
\]
The limiting pair $(u_{\varepsilon,m,\mu}, \lambda_{\varepsilon,m,\mu})$ is therefore a positive normalized weak solution of the singular $\varepsilon$-regularized problem
\begin{align}\label{eq:singular-limit-weak-equation}
\int_\Omega A_\varepsilon \left(x,\nabla u_{\varepsilon,m,\mu} \right) \cdot\nabla\phi dx &= \lambda_{\varepsilon,m,\mu} \int_\Omega u_{\varepsilon,m,\mu}^{p(x)-1} \phi dx + \mu \int_\Omega h u_{\varepsilon,m,\mu}^{-\gamma(x)} \phi  dx,\quad J(u_{\varepsilon,m,\mu})=m
\end{align}
for every $\phi\in X$.
\end{theorem}

\begin{proof}
Choose and fix a function $v\in\mathcal{M}_m^+$.  Such a function exists by the argument used in the proof of Lemma \ref{thm:approx}. Since $(u_n)$ minimizes $I_{\mu,\delta_n}$ over $\mathcal{M}_m^+$, we have
\[
I_{\mu,\delta_n}(u_n) \leq I_{\mu,\delta_n}(v)
\]
for every $n$. Expanding the definition of the functional gives
\[
E_\varepsilon(u_n) - \mu \int_\Omega h(x)G_{\delta_n}(x,u_n)dx \leq E_\varepsilon(v) - \mu \int_\Omega h(x)G_{\delta_n}(x,v)dx.
\]
Since $h(x)G_{\delta_n}(x,v)\geq0$ for a.e. in $\Omega$, we have
\[
E_\varepsilon(u_n) \leq E_\varepsilon(v) + \mu \int_\Omega h(x)G_{\delta_n}(x,u_n)dx.
\]
The lower growth estimate for $E_{\varepsilon}$ in Lemma \ref{lem:energy}, and Lemma \ref{lem:primitive} gives
\[
\rho_{p(\cdot)}(\nabla u_n) \leq p^+ \left(E_\varepsilon(v) + \mu C_{\mathrm{sing}} + C_1|\Omega| \right)
\]
for every $n$; that is,
\begin{equation}\label{theo3.3-a}
\sup_n \|u_n\|<\infty.
\end{equation}
Therefore there is a subsequence, not relabelled, and a function $ u:=u_{\varepsilon,m,\mu} \in X$ such that $ u_n\rightharpoonup u$ in $X$.  In particular, passing to a further subsequence if necessary, we obtain \eqref{eq:strongu}, \eqref{eq:weakgrad}, \eqref{eq:aeu}. Since $u_n\geq0$ a.e. in $\Omega$ for every $n$,  $u\geq0$ a.e. in $\Omega$. Moreover, by Lemma \ref{lem:constraint}, $J$ is continuous on $L^{p(\cdot)}(\Omega)$, hence $J(u_n)\rightarrow J(u)$. Since every $u_n$ belongs to $\mathcal{M}_m^+$, $J(u_n)=m$. It follows that $ J(u)=m$, and consequently, $u\in\mathcal{M}_m^+$.
On the other hand, by Assumption \ref{ass:barrier},
$u_n(x) \geq c_*d(x)^\vartheta$ for almost every $x\in\Omega$ and every sufficiently large $n$. Hence passing to the limit gives  $u(x)>0$ for almost every $x\in\Omega$. Thus the limiting singular expression $u(x)^{-\gamma(x)}$ is well defined almost everywhere.

\noindent
Define
\[
f_n(x):= h(x)(u_n(x)+\delta_n)^{-\gamma(x)} \quad \text{and} \quad f(x):= h(x)u(x)^{-\gamma(x)}.
\]
By Lemma \ref{lem:duality} \eqref{lem3.1-c}, \eqref{lem3.1-def},
\[
\sup_n \|f_n\|_ {*} <\infty,
\]
since this norm is independent of $\delta$. Thus, for every fixed $\phi\in X$, the DCT gives
\begin{equation}\label{theo3.3-c}
\int_\Omega f_n(x)\phi(x) dx \rightarrow \int_\Omega f(x)\phi(x) dx.
\end{equation}
On the other hand, since
\[
u_n(x)\rightarrow u(x),\, u(x)>0\, \text{ a.e. } x\in\Omega,\,\text{ and }\delta_n\rightarrow0
\]
it follows
\begin{equation}\label{theo3.3-d}
G_{\delta_n}(x,u_n(x)) = \frac{ (u_n(x)+\delta_n)^{1-\gamma(x)} - \delta_n^{1-\gamma(x)} }{ 1-\gamma(x) } \rightarrow \frac{u(x)^{1-\gamma(x)}}{1-\gamma(x)}
\end{equation}
for almost every $x\in\Omega$. To justify passage under the integral, we use the estimate \eqref{G-a}, i.e.
\[
0 \leq G_{\delta_n}(x,u_n) \leq \frac{u_n^{1-\gamma(x)}}{1-\gamma^+}.
\]
For every $a,b\geq0$ and every $q\in(0,1)$, it holds $|a^q-b^q| \leq |a-b|^q$. Applying this inequality pointwise with $q=1-\gamma(x)$, we obtain
\[
\left|u_n^{1-\gamma(x)} - u^{1-\gamma(x)} \right| \leq |u_n-u|^{1-\gamma(x)}.
\]
Because $0<1-\gamma(x)<p(x)$,  the same decomposition into the sets
\[
\{|u_n-u|\leq1\} \quad \text{and} \quad \{|u_n-u|>1\}
\]
gives
\begin{equation}\label{theo3.3-e}
|u_n-u|^{1-\gamma(x)} \leq 1+|u_n-u|^{p(x)}.
\end{equation}
Note that due to the strong convergence $ v_n:=u_n-u\rightarrow 0$ in $L^{p(\cdot)}(\Omega)$ we have $\rho_{p(\cdot)}(v_n) \to 0$. Fix $\varepsilon \in (0,1)$ and split the domain $\Omega$ into two sets for each $n$ as follows
\[
\Omega_{\leq \varepsilon}:= \{x \in \Omega : |v_n(x)| \leq \varepsilon\}, \quad \Omega_{> \varepsilon}:= \{x \in \Omega : |v_n(x)| > \varepsilon\}.
\]
For $x \in \Omega_{\leq \varepsilon}$: By $1 -\gamma(x) \geq 1-\gamma^+ > 0$, we obtain
\[
|v_n(x)|^{1-\gamma(x)} \leq \varepsilon^{1-\gamma(x)} \leq \varepsilon^{1-\gamma^+}.
\]
Integrating over $\Omega_{\leq \varepsilon}$ gives
\begin{equation}\label{theo3.3-f}
\int_{\Omega_{\leq \varepsilon}} |v_n|^{1-\gamma(x)} dx \leq \varepsilon^{1-\gamma^+} |\Omega|.
\end{equation}
For $x \in \Omega_{> \varepsilon}$: By $1-\gamma(x) <p(x)$, we have
\[
|v_n|^{1-\gamma(x)} = |v_n|^{p(x)} |v_n|^{1-\gamma(x)-p(x)} \leq \varepsilon^{1-\gamma^+-p^+} |v_n(x)|^{p(x)}.
\]
Integrating over $\Omega_{> \varepsilon}$ gives
\begin{equation}\label{theo3.3-g}
\int_{\Omega_{> \varepsilon}} |v_n|^{1-\gamma(x)} dx \leq \varepsilon^{1-\gamma^+-p^+} \int_\Omega |v_n|^{p(x)} dx.
\end{equation}
Combining \eqref{theo3.3-f} and \eqref{theo3.3-g} gives
\[
\int_\Omega |v_n|^{1-\gamma(x)} dx \leq \varepsilon^{1-\gamma^+} |\Omega| + \varepsilon^{1-\gamma^+-p^+} \int_\Omega |v_n|^{p(x)} dx.
\]
Taking $n \to \infty$ eliminates the second term and yields
\[
\limsup_{n \to \infty} \int_\Omega |v_n(x)|^{1-\gamma(x)} dx \leq \varepsilon^{1-\gamma^+} |\Omega|.
\]
Since $\varepsilon > 0$ was arbitrary, sending $\varepsilon\to 0^+$ establishes
$\int_\Omega |v_n|^{1-\gamma(x)} dx \to 0$, which implies
\begin{equation}\label{theo3.3-h}
u_n^{1-\gamma(x)} \rightarrow u^{1-\gamma(x)} \,\text{ in }L^1(\Omega) .
\end{equation}
We now split the difference as follows:
\[
\left|G_{\delta_n}(x,u_n) - \frac{u^{1-\gamma(x)}}{1-\gamma(x)} \right| \leq \frac{ \left|(u_n+\delta_n)^{1-\gamma(x)} - u_n^{1-\gamma(x)} \right| }{ 1-\gamma(x) } + \frac{ \delta_n^{1-\gamma(x)} }{ 1-\gamma(x) } + \frac{ \left|u_n^{1-\gamma(x)} - u^{1-\gamma(x)} \right| }{1-\gamma(x)}.
\]
Since $0<1-\gamma(x)<1$, subadditivity gives $(u_n+\delta_n)^{1-\gamma(x)} - u_n^{1-\gamma(x)} \leq \delta_n^{1-\gamma(x)}$. Therefore,
\[
\left|G_{\delta_n}(x,u_n) - \frac{u^{1-\gamma(x)}}{1-\gamma(x)} \right| \leq \frac{ 2\delta_n^{1-\gamma(x)} + \left|u_n^{1-\gamma(x)} - u^{1-\gamma(x)} \right|}{ 1-\gamma(x) } \leq \frac{ 2\delta_n^{1-\gamma^+} + \left|u_n^{1-\gamma(x)} - u^{1-\gamma(x)} \right|}{1-\gamma^+ },
\]
because $0<\delta_n\leq1$ and $1-\gamma(x)\geq1-\gamma^+$. Multiplying by $h$ and integrating, we obtain
\[
\int_\Omega h(x) \left|G_{\delta_n}(x,u_n) - \frac{u^{1-\gamma(x)}}{1-\gamma(x)} \right| dx \leq \frac{\|h\|_\infty}{1-\gamma^+} \left(2|\Omega|\delta_n^{1-\gamma^+} + \int_\Omega \left| u_n^{1-\gamma(x)}- u^{1-\gamma(x)} \right|dx \right).
\]
Both terms on the right tend to zero as $n \to\infty$ providing
\begin{equation}\label{theo3.3-i}
\int_\Omega h(x)G_{\delta_n}(x,u_n(x))dx \rightarrow \int_\Omega h(x) \frac{u(x)^{1-\gamma(x)}}{1-\gamma(x)}dx
\end{equation}
Next, define the limiting singular functional
\[
I_{\mu,0}(w) := E_\varepsilon(w) - \mu \int_\Omega h(x) \frac{w(x)^{1-\gamma(x)}}{1-\gamma(x)} dx, \quad w\in\mathcal{M}_m^+.
\]
This functional is finite on \(\mathcal{M}_m^+\) by Lemma~\ref{lem:primitive} and the corresponding limiting estimate. Let $w\in\mathcal{M}_m^+$ be arbitrary. Since $(u_n)$ minimizes $I_{\mu,\delta_n}$ over $\mathcal{M}_m^+$, we have $I_{\mu,\delta_n}(u_n) \leq I_{\mu,\delta_n}(w)$; that is,
\[
E_\varepsilon(u_n) - \mu \int_\Omega hG_{\delta_n}(x,u_n)dx \leq E_\varepsilon(w) -\mu \int_\Omega hG_{\delta_n}(x,w)dx.
\]
For fixed $w$, applying exactly the same convergence argument as above gives
\[
\int_\Omega hG_{\delta_n}(x,w)dx \rightarrow \int_\Omega h(x) \frac{w^{1-\gamma(x)}}{1-\gamma(x)}dx.
\]
Taking the lower limit and considering the weak lower semicontinuity of $E_\varepsilon$  gives
\[
I_{\mu,0}(u) = E_\varepsilon(u)-\mu \int_\Omega h(x) \frac{u^{1-\gamma(x)}}{1-\gamma(x)} dx \leq \liminf_{n\to\infty} I_{\mu,\delta_n}(u_n) \leq \lim_{n\to\infty} I_{\mu,\delta_n}(w) = I_{\mu,0}(w).
\]
However, since \(w\in\mathcal{M}_m^+\) was arbitrary, it follows
\begin{equation}\label{theo3.3-j}
I_{\mu,0}(u) = \inf_{w\in\mathcal{M}_m^+}I_{\mu,0}(w).
\end{equation}
We now show the convergence of the regularized energies $ E_\varepsilon(u_n)\rightarrow E_\varepsilon(u)$. Let $u\in\mathcal{M}_m^+$, then
\[
I_{\mu,\delta_n}(u_n) \leq I_{\mu,\delta_n}(u).
\]
Hence
\begin{equation}\label{theo3.3-k}
E_\varepsilon(u_n) \leq E_\varepsilon(u) + \mu \left(\int_\Omega hG_{\delta_n}(x,u_n)dx - \int_\Omega hG_{\delta_n}(x,u)dx\right).
\end{equation}
Considering \eqref{theo3.3-i}, and taking the upper limit in \eqref{theo3.3-k} gives
\[
\limsup_{n\to\infty}E_\varepsilon(u_n)\leq E_\varepsilon(u).
\]
Together with weak lower semicontinuity of $E_\varepsilon$, this proves the desired convergence
\begin{equation}\label{theo3.3-ka}
E_\varepsilon(u_n)\rightarrow E_\varepsilon(u).
\end{equation}
The convexity of $\Phi_\varepsilon(x,\cdot)$ gives, for almost every $x\in\Omega$,
\begin{equation}\label{theo3.3-l}
E_\varepsilon(u_n)-E_\varepsilon(u) \geq \int_\Omega A_\varepsilon(x,\nabla u) \cdot(\nabla u_n-\nabla u)dx.
\end{equation}
The right-hand side tends to zero because $A_\varepsilon(\cdot,\nabla u(\cdot)) \in L^{p'(\cdot)}(\Omega)$ and $(\nabla u_n-\nabla u) \rightharpoonup0 \,\text{ in }L^{p(\cdot)}(\Omega)$.
Now we apply the supporting-plane inequality with the roles of $\nabla u_n$ and $\nabla u$ interchanged; that is,
\[
\Phi_\varepsilon(x,\nabla u) \geq \Phi_\varepsilon(x,\nabla u_n) + A_\varepsilon(x,\nabla u_n)\cdot(\nabla u-\nabla u_n).
\]
Rearranging gives
\begin{equation}\label{theo3.3-m}
 A_\varepsilon(x,\nabla u_n) \cdot(\nabla u_n-\nabla u) \geq \Phi_\varepsilon(x,\nabla u_n)-\Phi_\varepsilon(x,\nabla u).
\end{equation}
Combining \eqref{theo3.3-l} and \eqref{theo3.3-m}, and taking the limit gives
\[
0 \leq \int_\Omega \left(A_\varepsilon(x,\nabla u_n) - A_\varepsilon(x,\nabla u) \right)\cdot(\nabla u_n-\nabla u)dx \rightarrow0.
\]
Equivalently,
\[
\left\langle \mathcal{A}_\varepsilon u_n - \mathcal{A}_\varepsilon u, u_n-u \right\rangle \rightarrow0.
\]
Since
\[
\left\langle \mathcal{A}_\varepsilon u, u_n-u \right\rangle \rightarrow 0
\]
we obtain
\[
\limsup_{n\to\infty} \left\langle \mathcal{A}_\varepsilon u_n, u_n-u \right\rangle \leq 0 .
\]
However, since $\mathcal{A}_\varepsilon$ is of type $(S_+)$, and $u_n\rightharpoonup u$  in $X$, this means
\begin{equation}\label{theo3.3-ma}
u_n\rightarrow u \text{ in } X.
\end{equation}
We now establish uniform bound for $\lambda_n$. Letting $\phi=u_n$ in \eqref{eq:approx-euler-limit-proof} and solving for $\lambda_n$ gives
\[
\lambda_n = \frac{ \displaystyle \int_\Omega A_\varepsilon(x,\nabla u_n) \cdot\nabla u_ndx - \mu \int_\Omega h(x)(u_n+\delta_n)^{-\gamma(x)} u_ndx }{ \displaystyle \int_\Omega u_n^{p(x)}dx}.
\]
Because
\[
J(u_n) = \int_\Omega \frac{u_n^{p(x)}}{p(x)}dx = m,
\]
we obtain
\[
p^-m \leq \int_\Omega u_n^{p(x)}dx \leq p^+m.
\]
For the principal term, the flux growth gives
\[
\left|A_\varepsilon(x,\nabla u_n) \cdot\nabla u_n \right| \leq |A_\varepsilon(x,\nabla u_n)| |\nabla u_n| \leq C  \left(|\nabla u_n| + |\nabla u_n|^{p(x)} \right).
\]
By the continuous embedding $L^{p(\cdot)}(\Omega) \hookrightarrow L^{1}(\Omega)$, there exists a constant $C>0$ such that
\begin{equation}\label{theo3.3-n}
\left|\int_\Omega A_\varepsilon(x,\nabla u_n) \cdot\nabla u_n\,dx \right| \leq C
\end{equation}
for every \(n\). For the singular term, since $u_n+\delta_n\geq u_n\geq 0$,
we have, at points where $u_n>0$,
\[
(u_n+\delta_n)^{-\gamma(x)}u_n \leq u_n^{-\gamma(x)}u_n = u_n^{1-\gamma(x)}.
\]
At points where $u_n=0$, the left-hand side is zero, so the same inequality remains valid. Therefore,
\[
0 \leq \int_\Omega h(x) (u_n+\delta_n)^{-\gamma(x)} u_ndx \leq \|h\|_\infty \int_\Omega u_n^{1-\gamma(x)}dx.
\]
As in Lemma~\ref{lem:primitive}, $u_n^{1-\gamma(x)} \leq 1+u_n^{p(x)}$. Thus
\[
\int_\Omega u_n^{1-\gamma(x)}dx \leq |\Omega| + \int_\Omega u_n^{p(x)}dx \leq |\Omega|+p^+m.
\]
Consequently,
\[
0 \leq \int_\Omega h(x)(u_n+\delta_n)^{-\gamma(x)}u_ndx \leq \|h\|_\infty \left(|\Omega|+p^+m\right).
\]
Combining the estimates for the numerator and denominator, we find a constant $C_\lambda>0$, independent of $n$, such that
\[
|\lambda_n|\leq C_\lambda.
\]
Since $(\lambda_n)$ is bounded in $\mathbb{R}$, there exists a subsequence, not relabelled, and a number $\lambda:=\lambda_{\varepsilon,m,\mu}\in\mathbb{R}$
such that
\begin{equation}\label{theo3.3-p}
\lambda_n\rightarrow \lambda.
\end{equation}
The convergence $\nabla u_n\rightarrow\nabla u$ in $L^{p(\cdot)}(\Omega)$ and the continuity of $A_\varepsilon(x, \cdot)$ gives
\[
A_\varepsilon(x,\nabla u_n) \rightarrow A_\varepsilon(x,\nabla u)\quad\text{in } L^{p'(\cdot)}(\Omega).
\]
Therefore, for every fixed $\phi\in X$,
\[
\left|\int_\Omega \left(A_\varepsilon(x,\nabla u_n) -A_\varepsilon(x,\nabla u) \right) \cdot\nabla\phi dx \right| \leq 2 \left\| A_\varepsilon(x,\nabla u_n)- A_\varepsilon(x,\nabla u) \right\|_{p'(\cdot)} \|\phi\| \rightarrow 0.
\]
Thus
\begin{equation}\label{theo3.3-r}
\int_\Omega A_\varepsilon(x,\nabla u_n)\cdot\nabla\phi dx \rightarrow \int_\Omega A_\varepsilon(x,\nabla u)\cdot\nabla\phi dx.
\end{equation}
Lastly, we show the convergence of the constraint term.\\
Define
\[
N(x,u):=|u(x)|^{p(x)-2}u(x).
\]
Then, clearly, $N$ is the Nemytskii operator, and hence, it is continuous and bounded operator from $L^{p(\cdot)}(\Omega)$ into $L^{p'(\cdot)}(\Omega)$. Since $u_n,u\geq0$,
\[
N(x,u_n)=u_n^{p(x)-1}, \qquad N(x,u)=u^{p(x)-1}.
\]
Therefore, strong convergence $ u_n\to u$ in $L^{p(\cdot)}(\Omega)$ along with  the continuity of $N(x,\cdot)$ leads to
\[
u_n^{p(x)-1} \rightarrow u^{p(x)-1} \,\text{ in }L^{p'(\cdot)}(\Omega).
\]
Hence, for every fixed \(\phi\in X\),
\[
\left|\int_\Omega \left(u_n^{p(x)-1} - u^{p(x)-1}\right)\phi dx \right| \leq 2 \left\|u_n^{p(\cdot)-1} - u^{p(\cdot)-1} \right\|_{p'(\cdot)} \|\phi\|_{p(\cdot)} \\ \rightarrow 0.
\]
Therefore,
\[
\int_\Omega u_n^{p(x)-1}\phi dx \rightarrow \int_\Omega u^{p(x)-1}\phi dx.
\]
Considering \eqref{theo3.3-p} this gives
\begin{equation}\label{theo3.3-s}
\lambda_n \int_\Omega u_n^{p(x)-1}\phi dx \rightarrow \lambda\int_\Omega u^{p(x)-1}\phi dx.
\end{equation}
Lastly, passing to the limit in \eqref{eq:approx-euler-limit-proof} and using \eqref{theo3.3-c}, \eqref{theo3.3-r}, and \eqref{theo3.3-s}  gives \eqref{eq:singular-limit-weak-equation}. The proof is complete.
\end{proof}

In the sequel, we present an example, a normalized singular problem on the unit disk, to verify the hypotheses \rm(H1)-\rm(H6). 

\begin{example}\label{ex:disk-singular-modular}
Let
\begin{equation}\label{eq:example-assump}
\Omega=B_1(0)\subset\mathbb R^2,\quad p(x)=2+|x|^2,\quad h(x)=1,\quad \gamma(x)=\frac12,\quad \varepsilon=m=\mu=1.
\end{equation}
Then the problem \eqref{eq:main-eigen-system} becomes
\begin{equation}\label{eq:disk-target}
\left\{
\begin{aligned}
-\mathrm{div} \left((1+|\nabla u|^2)^{|x|^2/2}\nabla u\right)
&=\lambda u^{1+|x|^2}+u^{-1/2} &&\text{in }\Omega,\\
u&>0 &&\text {in }\Omega,\\
u&=0 &&\text{on }\partial\Omega,\\
\int_\Omega\frac{u^{2+|x|^2}}{2+|x|^2} dx &=1.
\end{aligned}
\right.
\end{equation}

\noindent
We shall follow a step-by-step analysis.

\medskip
\noindent\textit{Step 1: approximate constrained minimizers.}
From \eqref{eq:example-assump}, it is easy to see that the conditions \rm(H1)-\rm(H4) holds. Next, let $X=W_0^{1,p(\cdot)}(\Omega)$ and
\[
J(v)=\int_\Omega\frac{|v|^{2+|x|^2}}{2+|x|^2}dx, \quad \mathcal M_1^+=\{v\in X:v \geq 0\text{ a.e.},\ J(v)=1\}.
\]
For $0<\delta\leq 1$, the approximate constrained functional is
\begin{equation}\label{eq:disk-functional}
I_{1,\delta}(v)=\int_\Omega\frac{(1+|\nabla v|^2)^{(2+|x|^2)/2}-1}{2+|x|^2} dx-2\int_\Omega\left(\sqrt{v+\delta}-\sqrt\delta\right)dx.
\end{equation}
Then Lemma \ref{thm:approx} guarantees the existence of a minimizer $u_\delta\in\mathcal M_1^+$. In particular, since $\sqrt{s}\leq1+s^{p(x)}$ for $s\geq 0$ and $|\Omega|=\pi$, one may obtain
\begin{equation}\label{eq:disk-mass-bounds}
2\leq\int_\Omega u_\delta^{p(x)} dx\leq 3 \quad \text{and} \quad \int_\Omega\sqrt{u_\delta} dx\leq \pi + 3.
\end{equation}

\medskip
\noindent\textit{Step 2: justification of the Euler equation.}
To avoid assuming that the nonnegative cone has an interior point at $u_\delta$, we extend the primitive to signed arguments by
\[
\widetilde G_\delta(s)=
\begin{cases}
2(\sqrt{s+\delta}-\sqrt\delta),
&s\geq 0,\\
\delta^{-1/2}s,&s<0.
\end{cases}
\]
Then $\widetilde G_\delta$ is a class of $C^1$ with bounded Lipschitz derivative $\widetilde G_\delta'(s)=(\max\{s,0\}+\delta)^{-1/2}$. Additionally, since $X\hookrightarrow L^2(\Omega)$, the functional
\[
\widetilde I_\delta(v)=E_1(v)-\int_\Omega\widetilde G_\delta(v) dx
\]
is $C^1$ on $X$, too. On the other hand, for every $v\in X$ with $J(v)=1$, we have
\[
E_1(|v|)=E_1(v),\quad J(|v|)= J(v),\quad \widetilde G_\delta(|v|)\geq \widetilde G_\delta(v),
\]
which provides
\[
\widetilde I_\delta(v)\geq I_{1,\delta}(|v|)\geq I_{1,\delta}(u_\delta) =\widetilde I_\delta(u_\delta).
\]
Thus $ u_\delta$ also minimizes the extended functional on the full level set $\{J=1\}$. Moreover, since $\langle J'(u_\delta),u_\delta\rangle \geq 2$, applying the Lagrange multiplier with $\lambda_\delta\in\mathbb R$ we obtain
\begin{align}\label{eq:disk-euler}
\int_\Omega A_1(x,\nabla u_\delta)\cdot\nabla\phi dx =\lambda_\delta\int_\Omega u_\delta^{p(x)-1}\phi dx +\int_\Omega(u_\delta+\delta)^{-1/2}\phi dx, \quad \phi\in X.
\end{align}
This extension leaves the original functional unchanged on nonnegative states and directly justifies its approximate weak equation. Indeed, testing with $u_\delta$ and using \eqref{eq:disk-mass-bounds} gives
\begin{equation}\label{eq:disk-multiplier-lower}
\lambda_\delta=\frac{\displaystyle\int_\Omega A_1(x,\nabla u_\delta)\cdot\nabla u_\delta dx-\int_\Omega\frac{u_\delta}{\sqrt{u_\delta+\delta}}dx}
{\displaystyle\int_\Omega u_\delta^{p(x)} dx}\geq-\frac{\pi+3}{2}=:-L,
\end{equation}
so, the principal term is nonnegative and $u_\delta/\sqrt{u_\delta+\delta}\leq\sqrt{u_\delta}$.

\medskip
\noindent\textit{Step 3: an explicit barrier independent of $\delta$.}
Set
\[
a:=\frac1{15+\pi}=\frac1{2(6+L)},\quad \text{and} \quad w(x):=a(1-|x|^2).
\]
Then $0<a<1/4$, $w\in X$, and $w=0$ on $\partial\Omega$. Writing $s=|x|^2$ and $F_a(s)=(1+4a^2s)^{s/2}$, then the direct differentiation of $A_1(x,\nabla w)=-2a F_a(s) x $ yields
\begin{equation}\label{eq:disk-barrier-computation}
-\mathrm{div} A_1(x,\nabla w)= 2aF_a(s)\left(2+s\log(1+4a^2s)+\frac{4a^2s^2}{1+4a^2s}\right)\leq 6a,
\end{equation}
where we used $0\leq s\leq1$, $4a^2\le1/4$, $F_a(s)\le\sqrt{5/4}<6/5$, and the fact that $\log(1+t)\leq t$ to obtain the last inequality.
Next, since $0\leq w\leq a<1$ and $p(x)-1\geq 1$, we obtain
\[
-\mathrm{div} A_1(x,\nabla w)+Lw^{p(x)-1}\leq (6+L)a=\frac12\le(w+1)^{-1/2}.
\]
On the other hand, \eqref{eq:disk-euler} and \eqref{eq:disk-multiplier-lower} imply, in the weak sense, that
\[
-\mathrm{div} A_1(x,\nabla u_\delta)+Lu_\delta^{p(x)-1}\geq(u_\delta+1)^{-1/2}.
\]
Thus subtracting these inequalities and testing with $\psi=(w-u_\delta)^+\in X$ gives
\begin{align*}
&\int_\Omega[A_1(x,\nabla w)-A_1(x,\nabla u_\delta)]\cdot\nabla\psi dx+L\int_\Omega[w^{p(x)-1}-u_\delta^{p(x)-1}]\psi dx\\
&\leq\int_\Omega[(w+1)^{-1/2}-(u_\delta+1)^{-1/2}] \psi dx\leq 0.
\end{align*}
The left-hand terms are nonnegative by monotonicity, while the second term is strictly positive wherever $\psi>0$, since $L>0$. Hence $\psi=0$ a.e., and $u_\delta\geq w$. Lastly, since $d(x)=1-|x|$, we have
\begin{equation}\label{eq:disk-uniform-barrier}
u_\delta(x)\ge\frac{1-|x|^2}{15+\pi} \geq \frac{d(x)}{15+\pi} \, \text{ a.e. in }\Omega,\, 0<\delta\leq 1.
\end{equation}
Thus condition \rm(H5) holds with $\delta_0=1$, $\vartheta=1$, and $c_*=(15+\pi)^{-1}$.

\medskip
\noindent\textit{Step 4: the boundary weight and global testing.}
Since $p(x)=2+|x|^2$, the corresponding conjugate exponent is $p'(x)=(2+|x|^2)/(1+|x|^2)$, and the boundary weight is
\[
\omega_1(x)=h(x)d(x)^{1-\gamma(x)}=d(x)^{1/2} \in L^\infty(\Omega)\subset L^{p'(\cdot)}(\Omega).
\]
Consequently, \rm(H6) also holds. More explicitly, by \eqref{eq:disk-uniform-barrier}
\[
(u_\delta+\delta)^{-1/2}\leq \sqrt{15+\pi}\, d^{-1/2},
\]
and therefore the Hardy and H\"older inequalities provide 
\begin{equation}\label{eq:Hardy-Holder-est}
\int_\Omega (u_\delta+\delta)^{-1/2}|\phi| dx \leq 2\sqrt{15+\pi}\,C_H\|d^{1/2}\|_{p'(\cdot)} \|\nabla\phi\|_{p(\cdot)},\,\, \phi\in X.
\end{equation}
Therefore all the hypotheses \rm(H1)-\rm(H6) are satisfied. 
\end{example}

\newpage
\section*{Conflict of interest}
The author declares that he has no conflict of interest.

\section*{Funding}
This work was supported by Athabasca University Research Incentive Account [140111 RIA].

\section*{ORCID}
\url{https://orcid.org/0000-0002-6001-627X}

\end{document}